\documentclass[11pt]{amsart}
\usepackage{amssymb,amsmath,amsthm,anysize}
\usepackage{longtable}
\usepackage{booktabs}
\usepackage{multirow,bigstrut,bigdelim,color}
\usepackage{lscape}
\usepackage{enumerate}
\usepackage{tikz}
\usetikzlibrary{graphs}
\usetikzlibrary{graphs.standard}
\usepackage{pgf,tikz}
\usepackage{mathrsfs}
\usetikzlibrary{arrows}
\usepackage{wrapfig}
 \usepackage{tikz}
\usetikzlibrary{graphs}
\usetikzlibrary{graphs.standard}
\usetikzlibrary{calc,intersections,through,backgrounds}
\usepackage{tkz-euclide}
\usepackage{graphicx}
\usepackage{amscd}
\usepackage{url}
\usepackage{eurosym}
\usepackage{verbatim}

\definecolor{ududff}{rgb}{0.30196078431372547,0.30196078431372547,1}

\newtheorem{theorem}{Theorem}[section]
\newtheorem{cor}[theorem]{Corollary}
\newtheorem{definition}[theorem]{Definition}
\newtheorem{notation}[theorem]{Notation}
\newtheorem{prop}[theorem]{Proposition}
\newtheorem{lemma}[theorem]{Lemma}
\newtheorem{remark}[theorem]{Remark}

\newcommand{\npmatrix}[1]{\left [ \begin{matrix} #1 \end{matrix} \right]}
\newcommand{\pt}[1]{\langle #1 \rangle}

\newcommand{\cV}{\mathcal{V}}
\newcommand{\cC}{\mathcal{C}}
\newcommand{\cP}{\mathcal{P}}
\newcommand{\cE}{\mathcal{E}}
\newcommand{\cI}{\mathcal{I}}
\newcommand{\bF}{\mathbb{F}}

\newcommand{\cZ}{\mathcal{Z}}
\newcommand{\cN}{\mathcal{N}}

\newcommand{\cH}{\mathcal{H}}

\newcommand{\GL}{\mathrm{GL}}

\newcommand{\PGL}{\mathrm{PGL}}

\newcommand{\PG}{\mathrm{PG}}

\renewcommand\le{\leqslant}
\renewcommand\ge{\geqslant}

\title{Nets of conics of rank one in $\PG(2,q)$, $q$ odd\\}

\author{Michel Lavrauw, Tomasz Popiel, John Sheekey}
\address{Michel Lavrauw, Sabanc{\i} University, Istanbul, Turkey; Email: \textsc{mlavrauw@sabanciuniv.edu}}
\address{Tomasz Popiel, Queen Mary University of London, UK; Email: \textsc{t.popiel@qmul.ac.uk}}
\address{John Sheekey, University College Dublin, Ireland; Email: \textsc{john.sheekey@ucd.ie}}
\date{\today}

\begin{document}

\begin{abstract}
We classify nets of conics of rank one in Desarguesian projective planes over finite fields of odd order, namely, two-dimensional linear systems of conics containing a repeated line. 
Our proof is geometric in the sense that we solve the equivalent problem of classifying the orbits of planes in $\PG(5,q)$ which meet the quadric Veronesean in at least one point, under the action of $\PGL(3,q) \le \PGL(6,q)$ (for $q$ odd). 
Our results complete a partial classification of nets of conics of rank one obtained by A.~H.~Wilson in the article ``The canonical types of nets of modular conics'', {\em American Journal of Mathematics} {\bf 36} (1914) 187--210.
\end{abstract}

\maketitle

%SECTION
\section{Introduction}
The space of forms of degree $d$ on an $n$-dimensional projective space $\PG(V)$ comprise a vector space $W$ of dimension ${n+d}\choose{d}$. 
Subspaces of the projective space $\PG(W)$ are called {\em linear systems of hypersurfaces of degree $d$}. 
The problem of classifying linear systems consists of determining the orbits of such subspaces under the induced action of the projectivity group $\PGL(V)$ on $\PG(W)$. 
One-dimensional linear systems are called a {\em pencils} and two-dimensional linear systems are called {\em nets}. 
In this paper we are concerned with linear systems of conics, namely the case $d=n=2$. 
Pencils of conics over $\mathbb{C}$ and $\mathbb{R}$ were classified by Jordan~\cite{Jordan1906, Jordan1907} in 1906--1907, and nets of conics over these fields were treated by Wall~\cite{Wall1977}. 
For an elementary exposition of the more general case $d=2$ (namely, pencils of quadrics), we refer the reader to chapters 9 and 11 of \cite{Casas-Alvero2010}.

Here we are concerned with linear systems of conics over {\em finite} fields. 
Compared with working over $\mathbb{C}$, complications arise when working over a finite field $\mathbb{F}_q$ (of order $q$) because $\mathbb{F}_q$ is not algebraically closed, resulting in a number of extra orbits which sometimes turn out to be difficult to classify. 
Pencils of conics over $\mathbb{F}_q$, $q$ odd, were classified by Dickson~\cite{Dickson1908}; an incomplete classification for $q$ even was obtained by Campbell~\cite{Campbell1927}. 
Our aim is to classify nets of conics over $\mathbb{F}_q$, $q$ odd.
This problem was addressed by Wilson~\cite{Wilson1914} using a purely algebraic approach, where the (in)equivalence of nets is studied by means of explicit coordinate transformations. 
However, as explained in Section~\ref{sec:Wilson}, Wilson's classification was incomplete. 
We take a geometric approach, based on the observation that linear systems of conics correspond to subspaces of $\PG(5,q)$. 
We classify the nets of {\em rank one}, namely those containing a repeated line, which correspond to planes in $\PG(5,q)$ intersecting the quadric Veronesean in at least one point. 
Our main result is Theorem~\ref{mainthm}. 
Here points of $\PG(5,q)$ are represented by symmetric $3 \times 3$ matrices, with the quadric Veronesean $\mathcal{V}(\mathbb{F}_q)$ defined by setting all $2 \times 2$ minors equal to zero (see Section~\ref{sec:prelims}).

\begin{theorem} \label{mainthm}
Let $q$ be a power of an odd prime. 
There are 15 orbits of planes in $\operatorname{PG}(5,q)$ that meet the quadric Veronesean in at least one point, under the action of $\PGL(3,q) \le \PGL(6,q)$ defined in Section~\ref{ss:PGL3q-action}. 
Representatives of these orbits are listed in Table~\ref{table:main}.
\end{theorem}

As a corollary, we complete Wilson's classification of nets of rank one, rectifying some of the statements made in his paper (see Section~\ref{sec:Wilson} for the details).

\begin{cor} \label{maincor}
There are 15 orbits of nets of conics of rank one in $\operatorname{PG}(2,q)$, $q$ odd.
\end{cor}

Our geometric approach provides insight which may be of use for other classification problems and is expected to have further applications in finite geometry. 
In particular, we are able to deduce further details about the plane orbits, such as the point-orbit distributions (see Definition~\ref{def:dist} and Table~\ref{table:pt-orbit-dist}), which serve as important invariants.
Data of this kind previously obtained by the first and second authors \cite{LaPo2017} for {\em lines} in $\operatorname{PG}(5,q)$ are used in the proof of Theorem~\ref{mainthm}. 

The paper is organised as follows.
In Section~\ref{sec:prelims} we give the necessary preliminaries for the proof of Theorem~\ref{mainthm}, in order to make the paper reasonably self-contained. The problem of classifying nets of conics in the classical projective plane $\PG(2,q)$, $q$ odd, is turned into the problem of classifying orbits of planes in 5-dimensional projective space under the action of a copy of the projectivity group $\PGL(3,q)$ viewed as a subgroup of $\PGL(6,q)$. The classification of points, lines, solids and hyperplanes in $\PG(5,q)$ is recalled in Section~\ref{sec:pts_lines_sols_hyps}.
Section~\ref{sec:planes} introduces some terminology, notation and lemmas used throughout the paper. 
The classification of planes in $\PG(5,q)$ (Theorem~\ref{mainthm}) is then proved in Sections~\ref{sec:three_rank_one_points}--\ref{sec:final}, with the proof divided into several parts.
In Section~\ref{sec:three_rank_one_points} we classify the planes spanned by three points of the quadric Veronesean, giving just two orbits, labelled $\Sigma_1$ and $\Sigma_2$. 
In Section~\ref{sec:two_rank_one_points} we classify the planes meeting the quadric Veronesean in exactly two points. There are three such orbits: $\Sigma_3$, $\Sigma_4$ and $\Sigma_5$.
The bulk of the classification is done in Section \ref{sec:one_rank_one_point_and_spanned}, which deals with the planes that are spanned by points of the secant variety of the quadric Veronesean and meet the quadric Veronesean in exactly one point. 
This leads to: eight further orbits $\Sigma_6, \ldots,\Sigma_{13}$ whose existence is independent of the characteristic of the field (as long as it is odd), one orbit $\Sigma_{14}$ which only appears in characteristic $\neq 3$, and one orbit $\Sigma_{14}'$ which only appears in characteristic~$3$.
In Section~\ref{sec:final}, we show that there is exactly one remaining orbit, $\Sigma_{15}$, consisting of planes that meet the Veronesean but are not spanned by points in the secant variety of the Veronesean.
Finally, in Section \ref{sec:Wilson}, we compare our classification with that of Wilson~\cite{Wilson1914} and deduce Corollary~\ref{maincor}.

%%% START PLANE ORBIT TABLE %%%
\begin{table}[!t]
\begin{tabular}{lll}
\toprule
Orbit & Representative & Conditions \\
\midrule
$\Sigma_1$ & $\left[ \begin{matrix} \alpha & \gamma & \cdot \\ \gamma & \beta & \cdot \\ \cdot & \cdot & \cdot \end{matrix} \right]$ & \\
\addlinespace[2pt]
$\Sigma_2$ & $\left[ \begin{matrix} \alpha & \cdot & \cdot \\ \cdot & \beta & \cdot \\ \cdot & \cdot & \gamma \end{matrix} \right]$ & \\
\addlinespace[2pt]
$\Sigma_3$ & $\left[ \begin{matrix} \alpha & \cdot & \gamma \\ \cdot & \beta & \cdot \\ \gamma & \cdot & \cdot \end{matrix} \right]$& \\
\addlinespace[2pt]
$\Sigma_4$ & $\left[ \begin{matrix} \alpha & \cdot & \gamma \\ \cdot & \beta & \gamma \\ \gamma & \gamma & \cdot \end{matrix} \right]$ & \\
\addlinespace[2pt]
$\Sigma_5$ & $\left[ \begin{matrix} \alpha & \cdot & \gamma \\ \cdot & \beta & \gamma \\ \gamma & \gamma & \gamma \end{matrix} \right]$ & \\
\addlinespace[2pt]
$\Sigma_6$ & $\left[ \begin{matrix} \alpha & \beta & \cdot \\ \beta & \varepsilon\alpha & \cdot \\ \cdot & \cdot & \gamma \end{matrix} \right]$ 
& $\varepsilon \in \boxtimes$ \\
\addlinespace[2pt]
$\Sigma_7$ & $\left[ \begin{matrix} \alpha & \beta & \gamma \\ \beta & \cdot & \cdot \\ \gamma & \cdot & \cdot \end{matrix} \right]$ & \\
\addlinespace[2pt]
$\Sigma_8$ & $\left[ \begin{matrix} \alpha & \beta & \cdot \\ \beta & \cdot & \gamma \\ \cdot & \gamma & \cdot \end{matrix} \right]$ & \\
\addlinespace[2pt]
\bottomrule
\end{tabular}
%%%
\begin{tabular}{lll}
\toprule
Orbit & Representative & Conditions \\
\midrule
$\Sigma_9$ & $\left[ \begin{matrix} \alpha & \beta & \cdot \\ \beta & \gamma & \cdot \\ \cdot & \cdot & -\gamma \end{matrix} \right]$ & \\
\addlinespace[2pt]
$\Sigma_{10}$ & $\left[ \begin{matrix} \alpha & \beta & \cdot \\ \beta & \gamma & \cdot \\ \cdot & \cdot & -\varepsilon\gamma \end{matrix} \right]$ & $\varepsilon \in \boxtimes$ \\
\addlinespace[2pt]
$\Sigma_{11}$ & $\left[ \begin{matrix} \cdot & \beta & \gamma \\ \beta & \alpha & \alpha \\ \gamma & \alpha & \alpha+\gamma \end{matrix} \right]$ & \\
\addlinespace[2pt]
$\Sigma_{12}$ & $\left[ \begin{matrix} \alpha & \beta & \cdot \\ \beta & \gamma & \beta \\ \cdot & \beta & \gamma \end{matrix} \right]$ & \\
\addlinespace[2pt]
$\Sigma_{13}$ & $\left[ \begin{matrix} \alpha & \beta & \cdot \\ \beta & \gamma & \beta \\ \cdot & \beta & \varepsilon\gamma \end{matrix} \right]$
& $\varepsilon \in \boxtimes$ \\
\addlinespace[2pt]
$\Sigma_{14}$ & $\left[ \begin{matrix} \alpha& \beta & \cdot \\ \beta & c\gamma & \beta-\gamma \\ \cdot & \beta-\gamma & \gamma \end{matrix} \right]$ 
& $q \not \equiv 0 \pmod 3$, ($\dagger$)
\\ 
\addlinespace[2pt]
$\Sigma_{14}'$ & $\left[ \begin{matrix} \alpha+\gamma & \gamma & \gamma \\ \gamma & \beta+\gamma & \gamma \\ \gamma & \gamma & -\beta \end{matrix} \right]$ 
& $q \equiv 0 \pmod 3$ \\
\addlinespace[2pt]
$\Sigma_{15}$ & $\left[ \begin{matrix} \alpha & \beta & \gamma \\ \beta & \gamma & \cdot \\ \gamma & \cdot & \cdot \end{matrix} \right]$ & \\
\addlinespace[2pt]
\bottomrule
\end{tabular}
%%%
\caption{Matrix representatives of the 15 orbits of planes in $\PG(5,q)$, $q$ odd, meeting the quadric Veronesean in at least one point, under the action of $\PGL(3,q) \le \PGL(6,q)$ defined in Section~\ref{ss:PGL3q-action}. 
Here $\cdot$ denotes $0$, $(\alpha,\beta,\gamma)$ ranges over all non-zero values in $\mathbb{F}_q^3$, and $\boxtimes$ is the set of non-squares in $\mathbb{F}_q$. 
In orbit $\Sigma_{14}$, condition~($\dagger$) is: $c \in \bF_q \setminus \{0,1\}$, $-3c \in \Box$ and $\tfrac{\sqrt{c}+1}{\sqrt{c}-1}$ is a not a cube in $\mathbb{F}_q(\sqrt{-3})$, where $\Box$ is the set of squares in $\mathbb{F}_q$.
}
\label{table:main}
\end{table}
%%% END PLANE ORBIT TABLE %%%

%SECTION
\section{Preliminaries}\label{sec:prelims}
In this section we review some of the theory used in the proof of Theorem~\ref{mainthm}. Most of it is well known and can be extracted from standard textbooks on projective and algebraic geometry, for example \cite{Casas-Alvero2010} and \cite{Harris}. Some parts are from \cite{LaPo2017}. 
For a survey of properties of Veronese varieties over fields with non-zero characteristic, we refer the reader to \cite{Havlicek}.

By $\bF_q$ we denote the finite field of order $q$, and we assume throughout that $q$ is odd. 
A {\em form} on a vector space $V$ (or the projective space $\PG(V)$) is a homogeneous polynomial in the polynomial ring over the coefficient field of $V$ in $\dim V$ variables. The zero locus in $\PG(V)$ of a form $f$ on $\PG(V)$ is denoted by $\cZ(f)$.

%SUBSECTION
\subsection{Nets of conics}
A ternary quadratic form $f$ on $\bF_q^3$ defines a {\em conic} $\cC=\cZ(f)$ in $\PG(2,q)$. 
Each two distinct conics $\cZ(f_1)$, $\cZ(f_2)$ define the {\em pencil of conics} 
\[\{\cZ(af_1+bf_2):a,b\in \bF_q, (a,b)\neq (0,0)\}.\]
Similarly, a {\em net of conics} $\cN$ is defined by three conics $\cC_i=\cZ(f_i)$ ($i=1,2,3$) in $\PG(2,q)$, not contained in a pencil:
\[
\cN=\{\cZ(af_1+bf_2+cf_3):a,b,c\in \bF_q, (a,b,c)\neq (0,0,0)\}.
\]
Given such a net of conics, one can consider
\begin{equation} \label{eqn:qform}
xf_1+yf_2+zf_3=a_{00}(x,y,z)X_0^2+a_{01}(x,y,z)X_0X_1+\dots + a_{22}(x,y,z)X_2^2
\end{equation}
as a quadratic form whose coefficients are linear forms in $x,y,z$. For each $a,b,c\in \bF_q$, not all zero, we obtain a conic 
\[
\cN(a,b,c)=\cZ(af_1+bf_2+cf_3).
\]
Since $q$ is odd we can consider the matrix $A_\cN$ of the bilinear form associated to the quadratic form (\ref{eqn:qform}). We define the {\em discriminant} of the net $\cN$ as 
\[
\Delta_\cN=\det (A_\cN).
\]
The discriminant $\Delta_\cN$ defines a cubic curve $\cZ(\Delta_\cN)$ in the plane $\PG(2,q)$.

\begin{lemma}
The conic $\cN(a,b,c)$ is singular if and only if $(a,b,c)$ lies on the cubic $\cZ(\Delta_\cN)$.
\end{lemma}

\begin{proof}
Since $q$ is odd, this follows from the fact that the matrix of the linear system obtained by setting the three partial derivatives of $\cN(a,b,c)$ evaluated at a point belonging to $\cN(a,b,c)$ equal to zero has determinant equal to $\Delta_\cN$ evaluated at $(a,b,c)$. 
\end{proof}

A net has {\em rank one} if it contains a repeated line; {\em rank two} if it contains no repeated lines but contains a conic which is not absolutely irreducible; and {\em rank three} if every conic in the net is absolutely irreducible.

%SUBSECTION
\subsection{The quadric Veronesean} \label{ss:V}
We represent points $y=(y_0,y_1,y_2,y_3,y_4,y_5,y_6)$ of $\PG(5,q)$ by symmetric matrices
\begin{equation} \label{eq:symmetricMatrix}
M_y= \left[ 
\begin{matrix} y_0 & y_1 & y_2 \\ y_1 & y_3 & y_4\\ y_2 & y_4 & y_5\\\end{matrix}
\right].
\end{equation}
The Veronese surface $\cV(\bF_q)$ in $\PG(5,q)$ is defined by setting the $2\times 2$ minors of the above matrix equal to zero, and we have the corresponding Veronese map from $\PG(2,q)$ to $\cV(\bF_q) \subset \PG(5,q)$: 
\[
\nu:(x_0,x_1,x_2)\mapsto (x_0^2,x_0x_1,x_0x_2,x_1^2,x_1x_2,x_2^2). 
\]
The {\em rank} of a point in $\PG(5,q)$ is defined to be the rank of the matrix $M_y$, and we denote by $\cP_i$ the set of points of rank $i$ for $i=1,2,3$. 
For convenience, we sometimes denote the rank of a point $x$ by $\operatorname{rank}(x)$. 
The points contained in $\cV(\bF_q)$ are points of rank $1$, and the points of rank $2$ are those contained in the secant variety of $\cV(\bF_q)$. 

Let us also partition the set of hyperplanes of $\PG(5,q)$ into the following subsets, depending on their intersection with $\cV(\bF_q)$:
\begin{itemize}
\item $\cH_1$ is the set of hyperplanes intersecting $\cV(\bF_q)$ in a conic,
\item $\cH_3$ is the set of hyperplanes intersecting $\cV(\bF_q)$ in a normal rational curve,
\item $\cH_2$ is the set of hyperplanes not contained in $\cH_1\cup \cH_3$.
\end{itemize}
Let $\delta$ denote the map from the set of conics in $\PG(2,q)$ to the set of hyperplanes in $\PG(5,q)$ which takes the conic $\cC=\cZ(f)$ with $f(X_0,X_1,X_2)=\sum_{i\le j}a_{ij}X_iX_j$ to
\[
\delta(\cC)=H[a_{00},a_{01},a_{02},a_{11},a_{12},a_{22}],
\]
where $H[a_0,\ldots,a_5]$ denotes the hyperplane with equation $a_0Y_0+\ldots +a_5Y_5=0$.

\begin{lemma} \label{nu-delta-lemma} 
The maps $\nu$ and $\delta$ satisfy the following properties, where $\cC$ is a conic in $\PG(2,q)$.
\begin{itemize}
\item[(i)] A point $x$ in $\PG(2,q)$ belongs to $\cC$ if and only if $\nu(x)\in \delta(\cC)$.
\item[(ii)] $\cC$ is a repeated line if and only if $\delta(\cC)\in \cH_1$.
\item[(iii)] $\cC$ is a point if and only if $\delta(\cC)\in \cH_2$ and $\delta(\cC)$ intersects $\cV(\bF_q)$ in a point.
\item[(iv)] $\cC$ is a union of two lines if and only if $\delta(\cC)\in \cH_2$ and $\delta(\cC)$ intersects $\cV(\bF_q)$ in two conics.
\item[(v)] $\cC$ is non-degenerate if and only if $\delta(\cC)\in \cH_3$.
\end{itemize}
\end{lemma}

Furthermore, the image of a line $\ell$ in $\PG(2,q)$ under the Veronese map is a conic in $\cV(\bF_q)$. 
A plane in $\PG(5,q)$ intersecting $\cV(\bF_q)$ in a conic is called a {\it conic plane}. 
Each conic plane is equal to $\langle \nu(\ell)\rangle$ for some line $\ell$ in $\PG(2,q)$. 
Each two points $x,y \in \cV(\bF_q)$ lie on a unique conic $\cC(x,y) \subset \cV(\bF_q)$ given by
\begin{equation} \label{eqn:C(x,y)}
\cC(x,y)=\nu(\langle \nu^{-1}(x),\nu^{-1}(y)\rangle).
\end{equation}
Each rank-$2$ point $z \in \langle \cV(\bF_q)\rangle$ lies in a unique conic plane $\langle \cC_z \rangle$. 
If $z$ is on the secant $\langle x,y\rangle$ then
\[
\cC_z = \cC(x,y).
\]

We also extend the definition of $\delta$ from the set of conics to the sets of pencils and nets of conics as follows. 
Given any set $\mathcal S$ of conics in $\PG(2,q)$, define
\[
\delta ({\mathcal{S}})=\cap_{\cC\in {\mathcal{S}}} \delta(\cC).
\]
In this way we obtain the following one-to-one correspondences.

\begin{lemma}
If $\cP$ is a pencil (respectively, net) of conics in $\PG(2,q)$ then $\delta(\cP)$ is a solid (respectively, plane) in $\PG(5,q)$, and conversely.
\end{lemma}

The dual map $\delta^*$ of $\delta$ from the set of conics in $\PG(2,q)$ to the set of points in $\PG(5,q)$ is given by
\[
\delta^*:\cZ\Big(\sum_{i\le j}a_{ij}X_iX_j\Big)\mapsto 
(a_{00},a_{01},a_{02},a_{11},a_{12},a_{22}).
\]
If $\cP$ is a pencil and $\cN$ is a net of conics in $\PG(2,q)$, then $\delta^*(\cP)$ is a line and $\delta^*(\cN)$ is a plane in $\PG(5,q)$.

%SUBSECTION
\subsection{The action of $\PGL(3,q)$ on $\PG(5,q)$} \label{ss:PGL3q-action}
Recall that we represent points in $\PG(5,q)$ by symmetric $3 \times 3$ matrices as per equation~\eqref{eq:symmetricMatrix} (modulo scalars). 
The action of the group $\PGL(3,q)$ on the points of $\PG(5,q)$ referred to in Theorem~\ref{mainthm} is defined as follows. 

If $\varphi_A \in \PGL(3,q)$ is induced by $A\in \GL(3,q)$ then we define $\alpha(\varphi_A) \in \PGL(6,q)$ by
\[
\alpha(\varphi_A):y\mapsto z \quad \text{where} \quad M_z=AM_yA^T.
\] 
We write
\[
K:=\alpha(\PGL(3,q))\le \PGL(6,q).
\]
This also defines an action of $\PGL(3,q)$ on subspaces of $\PG(5,q)$. 
The following observation is now readily deduced.

\begin{prop} \label{prop:nets_vs_planes}
The classification of nets of conics in $\PG(2,q)$ up to coordinate transformations is equivalent to the classification of the $K$-orbits of planes in $\PG(5,q)$.
\end{prop}

%SUBSECTION
\subsection{Properties of the non-degenerate conic in $\PG(2,q)$, $q$ odd}\label{subsec:conic_properties}

Let $\cC$ be a non-degenerate conic in $\PG(2,q)$, $q$ odd. 
Its projective stabiliser $G \le \PGL(3,q)$ is isomorphic to ${\mathrm{PGL}}(2,q)$. 
This can be seen using the Veronese map from $\PG(1,q)$ to $\PG(2,q)$, and it implies that
\begin{enumerate}
\item[(C1)] $G$ acts sharply 3-transitively on the points of $\cC$.
\end{enumerate}
The group $G$ has order $(q+1)q(q-1)$, and $\cC$ contains $q+1$ points. 
A point not on $\cC$ is an {\it external point} if it lies on a tangent line to $\cC$ (in which case it lies on exactly two tangents to $\cC$), and an {\it internal point} otherwise.
The set $\cE$ of external points has size $q(q+1)/2$, and the set $\cI$ of internal points has size $q(q-1)/2$.
If $x,y$ are two external points whose polar lines meet $\cC$ in points $x_1,x_2$ and $y_1,y_2$ respectively, then by mapping $x_i$ to $y_i$ for both $i=1,2$ we deduce that
\begin{enumerate}
\item[(C2)] $G$ acts transitively on $\cE$, and 
\item[(C2a)] the stabiliser $G_x \le G$ of an external point $x \in \cE$ acts transitively on the points of $\cC$ not on the polar line of $x$.
\end{enumerate}
In fact the above argument shows that one only needs 2-transitivity of $G$ on the points on $\cC$ to deduce transitivity on $\cE$, so we also have that
\begin{enumerate}
\item[(C3)] the stabiliser $G_w \le G$ of a point $w\in \cC$ acts transitively on both 
(i) the points of $\cE$ not on the tangent line $t_w(\cC)$ of $\cC$ at $w$, and (ii) the points of $t_w(\cC) \setminus \{w\}$. 
\end{enumerate}
If $x\in \cE$ then the orbit-stabiliser theorem implies that $|G_x|=2(q-1)$. 
The dual statements of the above properties are:
\begin{enumerate}
\item[(C4)] $G$ acts transitively on the secants of $\cC$, and
\item[(C5)] the stabiliser $G_w \le G$ of a point $w\in \cC$ acts transitively on both (i) the secants not passing through $w$, and (ii) the set of lines through $w$ different from $t_w(\cC)$.
\end{enumerate}

We now turn our attention to the action of $G$ on the set of internal points $\cI$.
For this we consider the quadratic extension $\bF_{q^2}$ of $\bF_q$, denoting the extension (if well defined) of an object $A$ over $\bF_q$ to $\bF_{q^2}$ by $\overline{A}$. 
An internal point $x$ of $\cC$ in $\PG(2,q)$ becomes an external point $\overline x$ of $\overline \cC$ in $\PG(2,q^2)$. 
If $\ell$ is the polar line of $x$ then $\overline \ell$ is a secant to $\overline \cC$. 
The stabiliser of $\overline \ell$ in $\overline G$ has order $2(q^2-1)$. 
The points of intersection of $\overline \ell$ and $\overline \cC$ are conjugate points $p$ and $p^\sigma$, where $\sigma$ denotes the involution in ${\mathrm{P\Gamma L}}(3,q^2)$ induced by the Frobenius map $a\mapsto a^q$. 
The stabiliser of $\ell\subset \overline \ell$ and the unordered pair $(p,p^\sigma)$ equals the stabiliser of $\ell$ in $G$, which is the group ${\mathrm{PGO}}^-(2,q)$. 
This group has order $2(q+1)$. 
In its action on $\overline \ell$ it acts sharply transitively on the points of $\ell$. 
This implies that the stabiliser $G_x$ of a point $x \in \cI$ has order $2(q+1)$, and hence that the orbit of $x$ under $G$ has size $q(q-1)/2 = |\cI|$. 
This implies that 

\begin{enumerate}
\item[(C6)] $G$ acts transitively on $\cI$, and
\item[(C7)] $G$ acts transitively on the set of lines external to $\cC$ (that is, not intersecting $\cC$).
\end{enumerate}

\begin{lemma}
Let $G$ be the projective stabiliser of a non-degenerate conic $\cC$ in $\PG(2,q)$, $q$ odd.
The stabiliser $G_w \le G$ of a point $w\in \cC$ acts transitively on the set of internal points to $\cC$.
\end{lemma}

\begin{proof}
We continue with the notation introduced above. 
Consider two distinct internal points $x,y$ with polar lines $\ell,m$. 
Let $p, p^\sigma$ and $r, r^\sigma$ be the points of intersection of $\overline \ell$ and $\overline m$ with $\overline \cC$.
Let $\alpha$ be the unique element of ${\mathrm{PGL}}(2,q^2)$ mapping the frame $(p',p'^\sigma,w')$ to the frame $(r',r'^\sigma,w')$, where $p', r' , \ldots $ denote the preimages of the points $p, r, \ldots$ under the Veronese map from $\PG(1,q^2)$ to $\overline \cC$. 
Then $\alpha \sigma \alpha^{-1}\sigma$ fixes the frame $(p',p'^\sigma,w')$ and hence, since $\alpha$ belongs to ${\mathrm{PGL}}(2,q^2)$, which is normal in ${\mathrm{P\Gamma L}}(2,q^2)$, it follows that $\sigma \alpha^{-1}\sigma=\sigma^{-1} \alpha^{-1}\sigma\in {\mathrm{PGL}}(2,q^2)$. 
Therefore, $\alpha \sigma \alpha^{-1}\sigma$ is the identity, so $\alpha$ commutes with $\sigma$, implying $\alpha \in {\mathrm{PGL}}(2,q)$.
It follows that $\alpha$ induces an element in $G_w$ mapping $\ell$ to $m$ and therefore also $x$ to $y$. 
\end{proof}

We summarise the above results in the following lemma (using again the same notation).

\begin{lemma}\label{lem:Gw}
Let $G$ be the projective stabiliser of a non-degenerate conic $\cC$ in $\PG(2,q)$, $q$ odd, and let $G_w \le G$ be the stabiliser of a point $w\in \cC$. 
Then the $G_w$-orbits of points in $\PG(2,q)$ are precisely
$\{w\}$, $\cC \backslash \{w\}$, $\cE\backslash t_w(\cC)$, $t_w(\cC)\backslash \{w\}$, and $\cI$.
\end{lemma}

Finally, we note that for $w,u \in \cC$ the two-point stabiliser $G_{w,u} \le G$ acts sharply transitively on points of $\cC\backslash \{w,u\}$ and has in total $q+6$ orbits on the points of $\PG(2,q)$:
\begin{itemize}
\item three fixed points $\{w\},\{u\},\{t_u\cap t_w\}$, 
\item the points of $\cC\backslash\{u,w\}$, 
\item the external points on the line $\pt{u,w}$, 
\item the internal points on the line $\pt{u,w}$, 
\item the points on $t_w(\cC) \backslash \{w,t_u(\cC)\cap t_w(\cC)\}$, 
\item the points on $t_u(\cC) \backslash \{u,t_u(\cC)\cap t_w(\cC)\}$, 
\item $q-2$ additional orbits of size $q-1$. 
\end{itemize}

%SECTION
\section{Points, lines, solids and hyperplanes of $\PG(5,q)$}\label{sec:pts_lines_sols_hyps}

As before, let $K$ denote the subgroup $\alpha(\PGL(3,q))$ of $\PGL(6,q)$ defined in Section~\ref{ss:PGL3q-action}. 
Before we study the classification of nets of conics, we consider the other linear systems of conics in $\PG(2,q)$, or equivalently, the $K$-orbits on subspaces of $\PG(5,q)$ of dimension $\neq 2$. 
A more general version of Proposition \ref{prop:nets_vs_planes} is the following.

\begin{lemma}\label{lem:duality}
If $q$ is odd then the $K$-orbits of points (respectively, lines) in $\PG(5,q)$ are in one-to-one correspondence with the $K$-orbits of hyperplanes (respectively, solids) in $\PG(5,q)$.
\end{lemma}

\begin{proof}
This is well known; see for example \cite[Section~6]{LaPo2017}.
\end{proof}

The $K$-orbits on points of $\PG(5,q)$ are well understood (see for example \cite[Section 2.2]{LaPo2017}).
In the notation established in Section~\ref{ss:V}, the correspondence between the $K$-orbits of points and the $K$-orbits of hyperplanes is as follows.

\begin{lemma}
For $i\in \{1,3\}$, the $K$-orbit $\cP_i$ corresponds to the $K$-orbit $\cH_i$. The $K$-orbits on $\cP_2$ correspond to the $K$-orbits on $\cH_2$.
\end{lemma}

Recall from Section~\ref{ss:V} that each rank-$2$ point $z \in \langle \cV(\bF_q)\rangle$ lies in a unique conic plane $\langle \cC_z \rangle$. 
If $z$ lies on a tangent to the conic $\cC_z$, it is said to be an {\em exterior} point; otherwise, it is said to be an {\em interior} point.  
Since $\bF_q$ is finite and $q$ is odd, there are two $K$-orbits on $\cP_2$: the orbit of exterior points $\cP_{2,e}$ and the orbit of interior points $\cP_{2,i}$ (see \cite[Section 2.2]{LaPo2017}).
There are also two $K$-orbits on hyperplanes in $\cH_2$: the orbit $\cH_{2,e}$ of hyperplanes intersecting $\cV(\bF_q)$ in two conics, and the orbit $\cH_{2,i}$ of hyperplanes intersecting $\cV(\bF_q)$ in a point. 

\begin{lemma} 
The $K$-orbits $\cP_{2,e}$ and $\cP_{2,i}$ are in one-to-one correspondence with the $K$-orbits $\cH_{2,e}$ and $\cH_{2,i}$, respectively.
\end{lemma}

\begin{remark}
\textnormal{ 
If $\bF_q$ is replaced by an algebraically closed field then each of $\cP_2$ and $\cH_2$ form one $K$-orbit. This is the origin of the complications which arise when working over finite fields.
}
\end{remark}

The $K$-orbits on lines in $\PG(5,q)$ were determined in \cite{LaPo2017}. 
For $q$ odd they are given by Theorem~\ref{linesThm} (with representatives listed in Table~\ref{table:lines}). 
The classification of $K$-orbits of solids in $\PG(5,q)$ then follows from Lemma~\ref{lem:duality}.

\begin{theorem} \label{linesThm}
There are 15 $K$-orbits on lines in $\PG(5,q)$, $q$ odd. 
Representatives of these orbits are listed in Table~\ref{table:lines}.
\end{theorem}

\begin{remark} \label{remOijNotation}
{\rm 
The notation for the line orbits, namely $o_i$ for $i\in \{5,6,9,10,12,16,17\}$ and $o_{i,j}$ for $i\in \{8,13,14,15\}$ and $j\in \{1,2\}$, is used for consistency with \cite[Table 2]{LaPo2017}. 
To be more specific, in \cite{LaPo2017} the orbits are denoted by $o_i$ for each $i$ as above, but when $i\in \{8,13,14,15\}$ there are in fact two `$o_i$' orbits. 
(The reason for the `$o_i$' notation itself is explained in \cite{LaPo2017}.) 
In the present paper, we instead explicitly label these orbits as $o_{i,j}$ where $j \in \{1,2\}$. 
In the case $i=15$, the representatives given in Table~\ref{table:lines} correspond to the representatives in column $j \in \{1,2\}$ of \cite[Table 2]{LaPo2017}. 
For $i \in \{8,13,14\}$, we have chosen slightly different representatives to those listed in \cite[Table 2]{LaPo2017}. 
The reason for this change is explained in Remark~\ref{remLineReps}.
}
\end{remark}

%%% START TABLE FROM LINES PAPER %%%
\begin{table}[!ht]
\begin{tabular}{lll}
\toprule
Orbit & Representative & Conditions \\
\midrule
$o_5$ & $\left[ \begin{matrix} \alpha & \cdot & \cdot \\ \cdot & \beta & \cdot \\ \cdot & \cdot & \cdot \end{matrix} \right]$ & \\
\addlinespace[2pt]
$o_6$ & $\left[ \begin{matrix} \alpha & \beta & \cdot \\ \beta & \cdot & \cdot \\ \cdot & \cdot & \cdot \end{matrix} \right]$ & \\
\addlinespace[2pt]
$o_{8,1}$ & $\left[ \begin{matrix} \alpha & \cdot & \cdot \\ \cdot & \beta & \cdot \\ \cdot & \cdot & -\beta \end{matrix} \right]$& \\
\addlinespace[2pt]
$o_{8,2}$ & $\left[ \begin{matrix} \alpha & \cdot & \cdot \\ \cdot & \beta & \cdot \\ \cdot & \cdot & -\varepsilon\beta \end{matrix} \right]$ 
& $\varepsilon \in \boxtimes$ \\
\addlinespace[2pt]
$o_9$ & $\left[ \begin{matrix} \alpha & \cdot & \beta \\ \cdot & \beta & \cdot \\ \beta & \cdot & \cdot \end{matrix} \right]$ & \\
\addlinespace[2pt]
$o_{10}$ & $\left[ \begin{matrix} v\alpha & \beta & \cdot \\ \beta & \alpha+u\beta & \cdot \\ \cdot & \cdot & \cdot \end{matrix} \right]$ 
& ($*$) \\
\addlinespace[2pt]
$o_{12}$ & $\left[ \begin{matrix} \cdot & \alpha & \cdot \\ \alpha & \cdot & \beta \\ \cdot & \beta & \cdot \end{matrix} \right]$ & \\
\addlinespace[2pt]
		%%% COMMENT OUT THE INDENTED BIT BELOW TO PUT THE TWO
		%%% HALF TABLES TOGETHER
		\addlinespace[26pt]
		\addlinespace[2pt]
		&& \\
		\bottomrule
		\end{tabular}
		%%%
		%%%
		\begin{tabular}{lll}
		\toprule
		Orbit & Representative & Conditions \\
		\midrule
		%%%
$o_{13,1}$ & $\left[ \begin{matrix} \cdot & \alpha & \cdot \\ \alpha & \beta & \cdot \\ \cdot & \cdot & -\beta \end{matrix} \right]$ & \\
\addlinespace[2pt]
$o_{13,2}$ & $\left[ \begin{matrix} \cdot & \alpha & \cdot \\ \alpha & \beta & \cdot \\ \cdot & \cdot & -\varepsilon\beta \end{matrix} \right]$ 
& $\varepsilon \in \boxtimes$ \\
\addlinespace[2pt]
$o_{14,1}$ & $\left[ \begin{matrix} \alpha & \cdot & \cdot \\ \cdot & -(\alpha+\beta) & \cdot \\ \cdot & \cdot & \beta \end{matrix} \right]$ & \\
\addlinespace[2pt]
$o_{14,2}$ & $\left[ \begin{matrix} \alpha & \cdot & \cdot \\ \cdot & -\varepsilon(\alpha+\beta) & \cdot \\ \cdot & \cdot & \beta \end{matrix} \right]$ 
& $\varepsilon \in \boxtimes$ \\
\addlinespace[2pt]
$o_{15,1}$ & $\left[ \begin{matrix} v\beta & \alpha & \cdot \\ \alpha & u\alpha+\beta & \cdot \\ \cdot & \cdot & \alpha \end{matrix} \right]$
& $-v \in \Box$, ($*$) \\
\addlinespace[2pt]
$o_{15,2}$ & $\left[ \begin{matrix} v\beta & \alpha & \cdot \\ \alpha & u\alpha+\beta & \cdot \\ \cdot & \cdot & \alpha \end{matrix} \right]$
& $-v \in \boxtimes$, ($*$) \\
\addlinespace[2pt]
$o_{16}$ & $\left[ \begin{matrix} \cdot & \cdot & \alpha \\ \cdot & \alpha & \beta \\ \alpha & \beta & \cdot \end{matrix} \right]$ & \\ 
\addlinespace[2pt]
$o_{17}$ & $\left[ \begin{matrix} v^{-1}\alpha & \beta & \cdot \\ \beta & u \beta - w \alpha & \alpha \\ \cdot & \alpha & \beta \end{matrix} \right]$ 
& ($**$) \\
\addlinespace[2pt]
\bottomrule
\end{tabular}
%%%
\caption{Matrix representatives of the 15 line orbits in $\PG(5,q)$, $q$ odd, under the action of $K=\PGL(3,q) \le \PGL(6,q)$ defined in Section~\ref{ss:PGL3q-action}. 
Here $\cdot$ denotes $0$, $(\alpha,\beta)$ ranges over all non-zero values in $\mathbb{F}_q^2$, and $\Box$ (respectively, $\boxtimes$) is the set of squares (respectively, non-squares) in $\mathbb{F}_q$. 
Condition~($*$) is: $v\lambda^2+uv\lambda - 1 \neq 0$ for all $\lambda \in \mathbb{F}_q$. 
Condition~($**$) is: $\lambda^3+w \lambda^2- u \lambda+ v \neq 0$ for all $\lambda \in \mathbb{F}_q$.
}
\label{table:lines}
\end{table}
%%% END TABLE FROM LINES PAPER %%%

The next lemma gives a useful condition for determining whether a rank-$2$ point in $\PG(2,q)$ is an exterior point or an interior point (that is, whether it belongs to the orbit $\cP_{2,e}$ or the orbit $\cP_{2,i}$). 
Here we use the notation $M_{ij}(A)$ to mean the matrix obtained from $A$ by removing the $i$th row and the $j$th column, and $|\cdot|$ denotes the determinant.

\begin{lemma}\label{lem:extcriteria}
A point $y\in \PG(5,q)$ belongs to $\cP_{2,e}$ if and only if $|M_y|=0$ and 
$-|M_{11}(M_y)|$, $-|M_{22}(M_y)|$, $-|M_{33}(M_y)|$
are all squares with at least one being non-zero.
\end{lemma}

\proof
Note that if the matrix $M_y$ defined by \eqref{eq:symmetricMatrix} satisfies the given conditions then $y$ is a point of rank $2$. 
Consider the rank-$2$ point $z_\tau$ with coordinates $(1,0,-\tau,0,0,0)$, where $\tau \in \bF_q\setminus\{0\}$, and write
\[
A_\tau=M_{z_\tau}=\npmatrix{1&0&0\\0&-\tau&0\\0&0&0}. 
\]
Note that $z_\tau$ is exterior for $\tau=1$ and interior for $\tau=\epsilon$ a non-square in $\bF_q$. 
A rank-$2$ point $y$ is therefore exterior (respectively, interior) if and only if there exists an invertible matrix $X$ such that $M_y=XA_{1}X^T$ (respectively, $M_y=XA_{\epsilon}X^T$). 
A straightforward calculation shows that for $M_y=XA_{\tau}X^T$ we have
\[
|M_{11}(M_y)| = -\tau|M_{13}(X)|^2, \quad
|M_{22}(M_y)| = -\tau|M_{23}(X)|^2, \quad
|M_{33}(M_y)| = -\tau|M_{33}(X)|^2,
\]
completing the proof. \qed

%SECTION
\section{Planes in $\PG(5,q)$, $q$ odd}\label{sec:planes}

We now collect a few final preliminaries before beginning the proof of Theorem~\ref{mainthm}. 
Recall that there are four $K$-orbits of points in $\PG(5,q)$, namely: $\cP_1$, the points of rank $1$; $\cP_{2,e}$, the exterior rank-$2$ points; $\cP_{2,i}$, the interior rank-$2$ points; and $\cP_3$, the points of rank $3$.

\begin{definition} \label{def:dist}
{\rm
Let $W$ be a subspace of $\PG(5,q)$. 
The {\em point-orbit distribution} of $W$ is the list $[n_1,n_2,n_3,n_4]$, where 
\begin{itemize}
\item $n_1$ is the number of rank-$1$ points contained in $W$,
\item $n_2$ is the number of exterior rank-$2$ points contained in $W$, 
\item $n_3$ is the number of interior rank-$2$ points contained in $W$, 
\item $n_4$ is the number of rank-$3$ points contained in $W$. 
\end{itemize}
The {\em rank distribution} of $W$ is the list $[m_1,m_2,m_3]$ where $m_i$ is the number of rank-$i$ points of $\PG(5,q)$ contained in $W$, for $i \in \{1,2,3\}$. 
In other words, $[m_1,m_2,m_3] = [n_1,n_2+n_3,n_4]$. 
Given $i \in \{1,2,3\}$, we say that $W$ has {\em constant rank} $i$ if $m_j = 0$ for all $j \neq i$. 
}
\end{definition}

The point-orbit distributions of the $K$-orbits of lines in $\PG(5,q)$ were previously determined by the first and second authors \cite[Tables~1 and~4]{LaPo2017}. 
They are summarised for reference in Table~\ref{table:line-rank-dist}. 
The point-orbit distributions of the $K$-orbits of planes are determined in a series of lemmas in Sections~\ref{sec:three_rank_one_points}--\ref{sec:final} and are summarised in Table~\ref{table:pt-orbit-dist}.

\begin{remark} \label{remLineReps}
{\rm 
As mentioned in Remark~\ref{remOijNotation}, we have chosen different representatives for the line orbits $o_{i,1}$ and $o_{i,2}$ in the cases $i \in \{8,13,14\}$, compared with \cite[Table 2]{LaPo2017}. 
The reason for this is that, in the representatives in \cite[Table 2]{LaPo2017}, the numbers of interior and exterior points of rank $2$ depend on whether $-1$ is a square in $\mathbb{F}_q$. 
This is explained in \cite[Table 4]{LaPo2017}. 
For the representatives given in Table~\ref{table:lines}, the point-orbit distributions are independent of $q$. 
}
\end{remark}

The following notation is used in Table~\ref{table:main} and throughout the proof of Theorem~\ref{mainthm}.

%%% START LINE point-orbit DIST TABLE %%%
\begin{table}[!t]
\begin{tabular}{ll}
\toprule
Orbit & Point-orbit distribution \\
\midrule
$o_5$ & $[2, \frac{q-1}{2}, \frac{q-1}{2}, 0]$ \\
\addlinespace[2pt]
$o_6$ & $[1, q, 0, 0]$ \\
\addlinespace[2pt]
$o_{8,1}$ & $[1, 1, 0, q-1]$ \\
\addlinespace[2pt]
$o_{8,2}$ & $[1, 0, 1, q-1]$ \\
\addlinespace[2pt]
$o_9$ & $[1, 0, 0, q]$ \\
\addlinespace[2pt]
$o_{10}$ & $[0, \frac{q+1}{2}, \frac{q+1}{2}, 0]$ \\
\addlinespace[2pt]
$o_{12}$ & $[0, q+1, 0, 0]$ \\
\addlinespace[2pt]
$o_{13,1}$ & $[0, 2, 0, q-1]$ \\
\addlinespace[2pt]
$o_{13,2}$ & $[0, 1, 1, q-1]$ \\
\addlinespace[2pt]
$o_{14,1}$ & $[0, 3, 0, q-2]$ \\
\addlinespace[2pt]
$o_{14,2}$ & $[0, 1, 2, q-2]$ \\
\addlinespace[2pt]
$o_{15,1}$ & $[0, 1, 0, q]$ \\
\addlinespace[2pt]
$o_{15,2}$ & $[0, 0, 1, q]$ \\
\addlinespace[2pt]
$o_{16}$ & $[0, 1, 0, q]$ \\
\addlinespace[2pt]
$o_{17}$ & $[0, 0, 0, q+1]$ \\
\addlinespace[2pt]
\bottomrule
\end{tabular}
\caption{Point-orbit distributions of $K$-orbits of lines in $\PG(5,q)$, $q$ odd.}
\label{table:line-rank-dist}
\end{table}
%%% END LINE point-orbit DIST TABLE %%%

%%% START POINT-ORBIT DISTRIBUTION TABLE %%%
\begin{table}[!t]
\begin{tabular}{lll}
\toprule
Orbit & Point-orbit distribution & Condition \\
\midrule
$\Sigma_1$ & $[q+1, q(q+1)/2, q(q-1)/2, 0]$ & \\
\addlinespace[2pt]
$\Sigma_2$ & $[3, 3(q-1)/2, 3(q-1)/2, q^2-2q+1]$ & \\
\addlinespace[2pt]
$\Sigma_3$ & $[2, (3q-1)/2, (q-1)/2, q^2-q]$ & \\
\addlinespace[2pt]
$\Sigma_4$ & $[2, (3q-1)/2, (q-1)/2, q^2-q]$ & \\
\addlinespace[2pt]
$\Sigma_5$ & $[2, q-1, q-1, q^2-q+1]$ & \\
\addlinespace[2pt]
$\Sigma_6$ & $[1, (q+1)/2, (q+1)/2, q^2-1]$ & \\
\addlinespace[2pt]
$\Sigma_7$ & $[1, q^2+q, 0, 0]$ & \\
\addlinespace[2pt]
$\Sigma_8$ & $[1, 2q, 0, q^2-q]$ & \\
\addlinespace[2pt]
$\Sigma_9$ & $[1, 2q, 0, q^2-q]$ & \\
\addlinespace[2pt]
$\Sigma_{10}$ & $[1, q, q, q^2-q]$ & \\
\addlinespace[2pt]
$\Sigma_{11}$ & $[1, q, 0, q^2]$ & \\
\addlinespace[2pt]
$\Sigma_{12}$ & $[1, (q-1)/2, (q-1)/2, q^2+1]$ & \\
\addlinespace[2pt]
$\Sigma_{13}$ & $[1, (q+1)/2, (q+1)/2, q^2-1]$ & \\
\addlinespace[2pt]
$\Sigma_{14}$ & $[1, (q\mp 1)/2, (q\mp 1)/2, q^2\pm 1]$ & $q \equiv \pm 1 \pmod 3$ \\
\addlinespace[2pt]
$\Sigma_{14}'$ & $[1, q, 0, q^2]$ & $q \equiv 0 \pmod 3$ \\
\addlinespace[2pt]
$\Sigma_{15}$ & $[1, q, 0, q^2]$ & \\
\addlinespace[2pt]
\bottomrule
\end{tabular}
\caption{Point-orbit distributions of $K$-orbits of planes in $\PG(5,q)$, $q$ odd.}
\label{table:pt-orbit-dist}
\end{table}
%%% END POINT-ORBIT DISTRIBUTION TABLE %%%

\begin{notation} \label{orbitNotation}
{\rm 
If a plane in $\PG(5,q)$ is spanned by points $x,y,z$ then we represent its $K$-orbit by the matrix $\alpha M_x +\beta M_y + \gamma M_z$, with the convention that zeroes are replaced by dots and the understanding that $(\alpha,\beta,\gamma)$ ranges over all non-zero values in $\mathbb{F}_q^3$.
For example, the $K$-orbit $\Sigma_1$ in equation~\eqref{sigma1and2} below (and the first row of Table~\ref{table:main}) is the $K$-orbit of the plane spanned by the points $x=(1,0,0,0,0,0)$, $y=(0,0,0,1,0,0)$ and $z=(0,1,0,0,0,0)$.
}
\end{notation}

We also note the following preliminary lemmas.

\begin{lemma}
Let $A$ be a $3\times 3$ matrix whose entries are linear forms in variables $x,y,z \in \mathbb{F}_q$. 
Then the points $(a,b,c)$ for which $A$ evaluated at $(a,b,c)$ has rank $1$ are singular points of the cubic $\cZ(|A|)$.
\end{lemma}

\begin{proof}
A $3 \times 3$ matrix has rank $1$ if and only if its adjugate is zero.
The partial derivative of the determinant of a matrix $M$ with polynomial entries with respect to a variable $x$ is given by
\[
\partial |M|/\partial x=|M| {\mathrm{tr}}(M^{-1}\partial M/\partial x).
\]
It follows that if $A$ evaluated at $(a,b,c)$ has rank $1$ then all of $\partial |A|/\partial x$, $\partial |A|/\partial y$, $\partial |A|/\partial z$ evaluated at $(a,b,c)$ are zero. This implies that $(a,b,c)$ is a singular point of $\cZ(|A|)$.
\end{proof}

For the next lemma, recall the notation for the $K$-orbits of lines in $\PG(5,q)$ from Theorem~\ref{linesThm}.

\begin{lemma} \label{lem:o6_plane} 
If $\pi$ is a plane in $\PG(5,q)$ containing a line of type $o_6$ and a point of rank $3$, then the cubic of points of rank at most $2$ in $\pi$ is either (i) the union of a non-degenerate conic and a tangent line, (ii) the union of a line and a double line, or (iii) a triple line.
\end{lemma}

\begin{proof}
By Table~\ref{table:line-rank-dist}, a line of type $o_6$ has rank distribution $[1,q,0]$.
Suppose that $\pi=\pt{x,y,z}$ with $\pt{x,y}$ a line of type $o_6$, $\operatorname{rank}(x)=1$, $\operatorname{rank}(y)=2$ and $\operatorname{rank}(z)=3$. 
Based on the $o_6$ representative in Table~\ref{table:lines}, we may assume that $x=(1,0,0,0,0,0)$ and $y=(0,1,0,0,0,0)$. 
The point $z$ then has coordinates $(0,0,a,b,c,d)$ for some $a,b,c,d \in \bF_q$.
The cubic $Q$ of points of rank at most $2$ has equation $X_3 f(X_1,X_2,X_3)=0$, where
\[
f(X_1,X_2,X_3)=(bd-c^2)X_1X_3-dX_2^2+2acX_2X_3-a^2bX_3^2.
\]
If the conic $\cC=\cZ(f)$ is non-degenerate then $d\neq 0$ and the line $\pt{x,y}$ is a tangent to $\cC$. 
If $\cC$ is degenerate then $d(bd-c^2)=0$.
If $d=0$ then $f(X_1,X_2,X_3)=X_3((bd-c^2)X_1+2acX_2-a^2bX_3)$ and $\cC$ is the union of two lines, at least one of which is $\pt{x,y}$.
If $d\neq0$ and $bd-c^2=0$ then $f(X_1,X_2,X_3)=-d^{-1}(dX_2-acX_3)^2$, and $\cC$ is a double line.
\end{proof}

Before finally proceeding to the classification of $K$-orbits of planes in $\PG(5,q)$, we prove the following lemma, which establishes the non-existence of planes with certain rank distributions.

\begin{lemma}\label{lem:b<q}
There are no planes in $\PG(5,q)$ with rank distribution $[2,n_2,n_3]$ where $n_2<q$.
\end{lemma} 

\begin{proof}
Consider any two distinct points $x,y \in \cV(\bF_q)$. 
Since every point on the line $\ell = \langle x,y\rangle$ has rank at most $2$ (as noted in Section~\ref{ss:V}), a plane $\pi$ through $x$ and $y$ contains at least $q-1$ points of rank $2$. 
Hence, if $\pi$ has rank distribution $[2,n_2,n_3]$ with $n_2<q$ then we must have $n_2=q-1$. 
It remains to rule out this case. 
Since all of the rank-$2$ points in $\pi$ lie on $\ell$, the cubic $Q$ of points of rank at most $2$ is the triple line $\ell$, because $x$ and $y$ are necessarily singular points of that cubic. 
Hence, if we write $\pi=\pt{x,y,z}$ then the point $z$ has rank $3$ and without loss of generality we may assume that $x=\nu(\pt{e_1})$, $y=\nu(\pt{e_2})$ and $z=(0,a,b,0,c,d)$ for some $a,b,c,d\in \bF_q$, so that $Q$ has equation $X_1X_3(dX_2-c^2X_3)+aX_3^3(ad-bc)+bX_3^2(acX_3-bX_2)=0$. 
Since this must be equivalent to $X_3^3=0$, it follows that $b=c=d=0$, contradicting $\operatorname{rank}(z)=3$. 
\end{proof}

%SECTION
\section{Planes spanned by three points of $\cV(\bF_q)$}\label{sec:three_rank_one_points}

\begin{notation}
{\rm 
Throughout the paper we let $e_1,e_2,e_3$ denote the standard basis vectors of $\bF_q^3$.
}
\end{notation}

Suppose that a plane $\pi$ in $\PG(5,q)$ contains three points $z_1,z_2,z_3 \in \cV(\bF_q)$, and consider their pre-images $p_1,p_2,p_3$ under the Veronese map $\nu : \PG(2,q) \rightarrow \cV(\bF_q)$. 
Since lines of $\PG(5,q)$ intersect $\cV(\bF_q)$ in at most two points, the points $z_1,z_2,z_3$ span $\pi$. 
If $p_1,p_2,p_3$ are collinear then without loss of generality $p_1=\langle e_1\rangle$, $p_2=\langle e_2\rangle$ and $\pi$ is the conic plane $\langle \cC(z_1,z_2)\rangle$. 
If $p_1,p_2,p_3$ are not collinear then without loss of generality $p_i=\langle e_i\rangle$ for $i=1,2,3$ and $\pi$ intersects $\cV(\bF_q)$ in the three points $z_1,z_2,z_3$. 
These two possibilities give us the following two $K$-orbits (respectively, with notation as established in Section~\ref{sec:planes}):
\begin{equation} \label{sigma1and2}
\Sigma_1: \left[ \begin{matrix} \alpha & \gamma & \cdot \\ \gamma & \beta & \cdot \\ \cdot & \cdot & \cdot \end{matrix} \right]
\quad \text{and} \quad
\Sigma_2 : \left[ \begin{matrix} \alpha & \cdot & \cdot \\ \cdot & \beta & \cdot \\ \cdot & \cdot & \gamma \end{matrix} \right].
\end{equation}

\begin{lemma}\label{lem:Sigma_1 and Sigma_2}
A plane belonging to the $K$-orbit $\Sigma_1$ has point-orbit distribution 
\[
[q+1,q(q+1)/2,q(q-1)/2,0],
\] 
and a plane belonging to the $K$-orbit $\Sigma_2$ has point-orbit distribution 
\[
[3,3(q-1)/2,3(q-1)/2,q^2-2q+1].
\]
\end{lemma} 

\begin{proof}
The point-orbit distribution of a plane in the $K$-orbit $\Sigma_1$ follows from Section \ref{subsec:conic_properties}. 
Let $\pi$ be a plane in the $K$-orbit $\Sigma_2$, and let $z_1,z_2,z_3$ denote the three points of rank $1$ in $\pi$. 
It is clear from the representative of $\Sigma_2$ that the points of rank $2$ in $\pi$ lie on the secants $\langle z_i,z_j\rangle$, $i\neq j$, which are lines of type $o_5$ by Table~\ref{table:lines}. 
It follows from a straightforward calculation and Lemma~\ref{lem:extcriteria} that each such secant contains $(q-1)/2$ external points and $(q-1)/2$ internal points of the conic $\cC(z_i,z_j)$ defined as in equation~\eqref{eqn:C(x,y)}. 
In total this amounts to $3(q-1)/2$ points of $\pi$ belonging to each of $\cP_{2,e}$ and $\cP_{2,i}$.
\end{proof}

%SUBSECTION
\section{Planes containing exactly two points of $\cV(\bF_q)$}\label{sec:two_rank_one_points}

We now classify the planes in $\PG(5,q)$ intersecting the quadric Veronesean in exactly two points. 
By Lemma \ref{lem:b<q}, we only need to consider planes with rank distribution $[n_1,n_2,n_2]$ where $n_1=2$ and $n_2\ge q$.
Let $\pi=\langle x,y,z\rangle$ be such a plane, with $\operatorname{rank}(x)=\operatorname{rank}(y)=1$.
The condition on the rank distribution of $\pi$ implies the existence of a point of rank $2$ which is not on the line $\langle x,y\rangle$. 
Hence, we may assume that $\operatorname{rank}(z)=2$.
Consider the conics $\cC(x,y)$ and $\cC_z$ (defined as in Section~\ref{ss:V}), and let $u = \cC(x,y)\cap\cC_z$.
There are two possibilities: $u\in \{x,y\}$ and $u\notin \{x,y\}$. 

First suppose that $u\in \{x,y\}$, say $u=x$. 
Then $\langle x,z\rangle$ is the unique tangent $t_x(\cC_z)$ to $\cC_z$ through $x$ in the plane $\langle \cC_z\rangle$, because otherwise $\pi$ would contain a third point of $\cV(\bF_q)$. 
If $\ell_z$ is the pre-image of $\cC_z$ under $\nu$, then the stabiliser of $\nu^{-1}(x)$ and $\nu^{-1}(y)$ acts transitively on the points of $\ell_z\setminus \{x\}$, and so the stabiliser of $\{x,y\}$ and $\cC_z$ acts transitively on the points of $t_x(\cC_z)\setminus\{x\}$. 
Hence, there is exactly one orbit of such planes, which we call $\Sigma_3$. 
Taking $x=\nu(e_1)$, $y=\nu(e_2)$, and $\ell_z = \langle e_1,e_3\rangle$ gives the following representative:
\[
\Sigma_3: \left[ \begin{matrix} \alpha & \cdot & \gamma \\ \cdot & \beta & \cdot \\ \gamma & \cdot & \cdot \end{matrix} \right].
\]

\begin{lemma}\label{lem:Sigma_3}
A plane belonging to the $K$-orbit $\Sigma_3$ has point-orbit distribution 
\[
[ 2 , (3q-1)/2 , (q-1)/2 , q^2-q ].
\]
\end{lemma} 

\begin{proof}
There are $q-1$ points of rank $2$ on the line $\langle x,y \rangle$, half belonging to $\cP_{2,e}$ and half belonging to $\cP_{2,i}$. 
There are also $q$ points of $\cP_{2,e}$ on the tangent to $\cC_z$ at $u=x$, giving a total of $q+(q-1)/2 = (3q-1)/2$ points of $\pi$ in $\cP_{2,e}$. 
The remaining points of $\pi$ are all of rank $3$.
\end{proof}

Now suppose that $u\notin \{x,y\}$. 
We may fix $x$, $y$, $u$ and $\cC_z$, and hence $\cC(x,y)$. 
The group which pointwise stabilises $\cC(x,y)$ and also setwise stabilises $\cC_z$ acts on $\langle\cC_z\rangle$ as the stabiliser of $\cC_z$ and $u$ in $\PGL(\langle\cC_z\rangle)$. 
Since $q$ is odd, this group has three orbits on points of $\langle\cC_z\rangle \backslash\cC_z$ (by Lemma~\ref{lem:Gw}):
(i) the points on the tangent to $\cC_z$ through $u$, (ii) the other external points of $\cC_z$, and (iii) the internal points of $\cC_z$.
We may choose $x=\nu(e_1)$, $y=\nu(e_2)$, $u=\nu(e_1+e_2)$ , $\cC_z =\nu(\pt{e_1+e_2,e_3})$.
We now consider separately the cases in which $z$ lies in each of these orbits.

\bigskip

(i) If $z$ is a point on the tangent to $\cC_z$ through $u$ then we obtain the orbit
\[
\Sigma_4 : \left[ \begin{matrix} \alpha & \cdot & \gamma \\ \cdot & \beta & \gamma \\ \gamma & \gamma & \cdot \end{matrix} \right].
\]

\begin{lemma}\label{lem:Sigma_4}
A plane belonging to the $K$-orbit $\Sigma_4$ has point-orbit distribution 
\[
[ 2 , (3q-1)/2 , (q-1)/2 , q^2-q ].
\]
\end{lemma} 

\begin{proof}
The proof of Lemma~\ref{lem:Sigma_3} also applies here (except that now $u \neq x$).
\end{proof}

\begin{remark} 
\textnormal{
Although planes in $\Sigma_3$ and $\Sigma_4$ have the same point-orbit distribution, these $K$-orbits are distinct. 
This can be seen by observing that for $\pi \in \Sigma_3$ the $(q-1)/2$ points in $\cP_{2,i}$ lie on a tangent line to a conic of $\cV(\bF_q)$ through one of the two points of rank $1$ in $\pi$, while for a plane in $\Sigma_4$ this is not the case.
}
\end{remark}

(ii) If $z$ is an external point of $\cC_z$ not on the tangent to $\cC_z$ through $u$ then we may choose $z$ in order to obtain the following representative of a new orbit $\Sigma_5$:
\[
\Sigma_5 : \left[ \begin{matrix} \alpha & \cdot & \gamma \\ \cdot & \beta & \gamma \\ \gamma & \gamma & \gamma \end{matrix} \right].
\]

\begin{lemma}\label{lem:Sigma_5}
A plane belonging to the $K$-orbit $\Sigma_5$ has point-orbit distribution 
\[
[2,q-1,q-1,q^2-q+1].
\]
\end{lemma} 

\begin{proof}
Let $\pi$ be the plane given above, that is, points of $\pi$ have coordinates $(\alpha,0,\gamma,\beta,\gamma,\gamma)$ with $\alpha, \beta, \gamma \in \bF_q$ not all zero.
The cubic $\cZ(\Delta_\cN)$ of the net $\cN$ corresponding to $\pi$ is the union of a non-degenerate conic $\cC$ and a line $\ell$ secant to $\cC$. 
The two points in $\cC\cap \ell$ are the points $x$ and $y$ from above, and the other points in $\cC \cup \ell$ are all of rank $2$. 
This implies that the rank distribution of $\pi$ is $[2,2(q-1),q^2-q+1]$. 
It remains to show that half of the $2(q-1)$ points of rank $2$ are exterior points and half are interior points. 
First consider the $q-1$ points of rank $2$ on the line $\ell=\langle x,y\rangle$. 
Since half of these points are external points of the conic $\cC(x,y)$ determined by $x$ and $y$, and the other half are internal points of $\cC(x,y)$, we obtain $(q-1)/2$ points in each of the point orbits $\cP_{2,e}$ and $\cP_{2,i}$. 
Now consider the $q-1$ points of rank $2$ in $\cC\setminus \{x,y\}$. 
These points have coordinates $p_{\alpha,\beta} := (\alpha,0,1,\beta,1,1)$ with $\alpha\beta=\alpha+\beta$ and $\alpha,\beta \neq 1$. 
Lemma~\ref{lem:extcriteria} implies that $p_{\alpha,\beta}\in \cP_{2,e}$ if and only if $1-\alpha$, $1-\beta$ and $-\alpha\beta$ are all squares in $\mathbb{F}_q$ (note that they cannot all be zero). 
The condition $\alpha\beta=\alpha+\beta$ implies that these three elements are either all squares or all non-squares. 
In particular, $p_{\alpha,\beta}\in \cP_{2,e}$ if and only if $1-\alpha$ is a square, which occurs for half of the possible $q-1$ values of $\alpha$ (recalling that $\alpha \neq 1$). 
Hence, we obtain a further $(q-1)/2$ points in each of $\cP_{2,e}$ and $\cP_{2,i}$, as claimed.
\end{proof}

(iii) Finally, we show that if $z$ is an internal point of $\cC_z$ then the plane $\pi$ again belongs to the $K$-orbit $\Sigma_5$. In other words, this case does not yield a new orbit.

\begin{lemma}
Suppose that $\pi=\pt{x,y,z}$ is a plane in $\PG(5,q)$ with $\operatorname{rank}(x)=\operatorname{rank}(y)=1$ and $\operatorname{rank}(z)=2$, where $x,y \notin \cC_z$ and $z$ is an internal point of $\cC_z$. 
Then $\pi \in \Sigma_5$.
\end{lemma}

\begin{proof}
An argument analogous to the one in the proof of Lemma~\ref{lem:Sigma_5} shows that $\pi$ contains a point $z' \in \cP_{2,e}$ which is not on the line $\langle x,y \rangle$, so $\pi = \langle x,y,z' \rangle \in \Sigma_5$.
 \end{proof}

%SECTION
\section{Planes meeting $\cV(\bF_q)$ in one point and spanned by points of rank $\le 2$}\label{sec:one_rank_one_point_and_spanned}

Let $\pi=\pt{x,y,z}$ be a plane with $\operatorname{rank}(x)=1$ and $\operatorname{rank}(y)=\operatorname{rank}(z)=2$.
Consider the conics $\cC_y$ and $\cC_z$ determined by $y$ and $z$ (see Section~\ref{ss:V}), and let $p_x$ be a point in $\PG(2,q)$ and $\ell_y$, $\ell_z$ lines in $\PG(2,q)$ such that $x=\nu(p_x)$, $\cC_y=\nu(\ell_y)$ and $\cC_z=\nu(\ell_z)$.
If $x \in \cC_y=\cC_z$, then $\pi$ is a conic plane, and so lies in the orbit $\Sigma_1$. 
The remaining possibilities (up to symmetry) are as follows: (a) $x \notin \cC_y=\cC_z$, (b) $x = \cC_y\cap \cC_z$, (c) $x \in \cC_y\setminus\cC_z$, and (d) $x\notin \cC_y\cup \cC_z$.

%SUBSECTION
\subsection{Case {\bf (a)} }

If $x \notin \cC_y=\cC_z$ then $\langle y,z\rangle$ is a line in $\langle\cC_y\rangle$ external to $\cC_y$. 
We may fix $x$ and $\cC_y$. 
The group stabilising both $x$ and $\cC_y$ acts on $\langle\cC_y\rangle$ as the stabiliser of $\cC_y$ in $\PGL(\langle\cC_y\rangle)$. 
This group acts transitively on external lines (by property (C7) of Section~\ref{subsec:conic_properties}), and so we have just one orbit arising in this way. 
Since $q$ is odd we have the following representative: 
\[
\Sigma_6 : \left[ \begin{matrix} \alpha & \beta & \cdot \\ \beta & \varepsilon\alpha & \cdot \\ \cdot & \cdot & \gamma \end{matrix} \right], 
\quad \text{where } \epsilon \in \bF_q \text{ is a non-square}.
\]

\begin{lemma}\label{lem:Sigma_6}
A plane belonging to the $K$-orbit $\Sigma_6$ has point-orbit distribution 
\[
[ 1 , (q+1)/2 , (q+1)/2 , q^2-1 ].
\] 
Its points of rank $2$ lie on a line of type $o_{10}$. 
\end{lemma} 

\begin{proof}
The points of rank $2$ are precisely those with $\gamma=0$, namely the points on the line $\pt{y,z}$ (in the above representative). 
Since this is a line in $\langle\cC_y\rangle$ external to $\cC_y$, half of its points belong to each of $\cP_{2,e}$ and $\cP_{2,i}$. 
By Table~\ref{table:line-rank-dist}, we see that the line has type $o_{10}$.
\end{proof}

\begin{lemma}\label{lem:Sigma_6_2}
A plane belonging to the $K$-orbit $\Sigma_6$ does not contain a line of constant rank $3$. 
\end{lemma} 

\begin{proof}
This follows immediately from Lemma~\ref{lem:Sigma_6}, because every line must intersect the line of type $o_{10}$, which is of constant rank $2$. 
\end{proof}

\subsection{Case {\bf (b)} }
If $x = \cC_y\cap \cC_z$ then $\langle x,y\rangle$ is the unique tangent to $\cC_y$ through $x$ in $\langle \cC_y \rangle$, and $\langle x,z\rangle$ is the unique tangent to $\cC_z$ through $x$ in $\langle \cC_z \rangle$ (because otherwise $\pi$ would contain more than one point of rank $1$). 
Hence, $\pi=\langle x,y,z\rangle$ is completely determined by $\ell_y$ and $\ell_z$ (because $p_x =\ell_y\cap \ell_z$). 
Since $\PGL(3,q)$ acts transitively on pairs of lines meeting in a point, we obtain just one orbit in this way. 
A representative is a follows.
\[
\Sigma_7 : \left[ \begin{matrix} \alpha & \beta & \gamma \\ \beta & \cdot & \cdot \\ \gamma & \cdot & \cdot \end{matrix} \right].
\]

\begin{lemma}\label{lem:Sigma_7}
A plane belonging to the $K$-orbit $\Sigma_7$ is a tangent plane of $\cV(\bF_q)$ and has point-orbit distribution $ [ 1 , q^2+q,0,0 ]$.
\end{lemma} 

\begin{proof}
Since such a plane $\pi$ contains two tangents through its unique point $x$ of rank $1$, it follows that $\pi$ is the tangent plane of $\cV(\bF_q)$ at $x$. 
The lines through $x$ in $\pi$ are the tangents to the conics of $\cV(\bF_q)$ through $x$, and so all the rank-$2$ points in $\pi$ are contained in $\cP_{2,e}$.
\end{proof}

%SUBSECTION
\subsection{Case {\bf (c)}}
Now suppose that $x \in \cC_y \setminus \cC_z$. 
Then $\langle x,y\rangle$ must be the unique tangent $t_x(\cC_y)$ to $\cC_y$ through $x$ in $\langle \cC_y \rangle$. 
Write $w = \cC_y\cap \cC_z$, and $\nu(p_w) = w$.
Without loss of generality we may fix $p_x$, $p_w$ and $\ell_z$. 
The subgroup of $K$ stabilising $x$, $w$ and $\pt{\cC_z}$ has three orbits on points of $\pt{\cC_z}\backslash \cC_z$: (i) the points on the unique tangent $t_w(\cC_z)$ to $\cC_z$ through $w$ in $\langle \cC_z \rangle$, (ii) the external points of $\cC_z$ not on $t_w(\cC_z)$, (iii) the internal points of $\cC_z$. 
Hence, we obtain at most three $K$-orbits, and we show below that we obtain exactly three orbits $\Sigma_8$, $\Sigma_9$ and $\Sigma_{10}$ (see Remark~\ref{remarkSigma8910}). 
The representatives of these orbits given below are obtained by choosing $p_x=\pt{e_1}$, $p_w=\pt{e_2}$ and $\ell_z =\pt{e_2,e_3}$.

\bigskip

{\bf (c-i)} If $z$ lies on $t_w(\cC_z)$ then $\pi=\pt{t_x,z}$ and we obtain the orbit
\[
\Sigma_8 : \left[ \begin{matrix} \alpha & \beta & \cdot \\ \beta & \cdot & \gamma \\ \cdot & \gamma & \cdot \end{matrix} \right].
\]

\begin{lemma}\label{lem:Sigma_8}
A plane belonging to the $K$-orbit $\Sigma_8$ has point-orbit distribution $ [ 1 , 2q,0,q^2-q ]$. 
Its points of rank $2$ lie on two lines: a tangent to $\cV(\bF_q)$, and a line of type $o_{12}$.
\end{lemma} 

\begin{proof}
Consider the plane $\pi=\pt{x,y,z}$ as above (where $y$ corresponds to $\alpha=\gamma=0$ and $z$ to $\alpha=\beta=0$).
The $q$ points other than $x$ on the line $\pt{x,y}$ belong to $\cP_{2,e}$. 
The other points of rank $2$ are the points on the line $\pt{y,z}$, which belongs to the $K$-orbit $o_{12}$ by Table~\ref{table:lines}. 
\end{proof}

{\bf (c-ii) and (c-iii)} If $z$ is an external point of $\cC_z$ but does not lie on $t_w(\cC_z)$ then we obtain the orbit
\[
\Sigma_9 : \left[ \begin{matrix} \alpha & \beta & \cdot \\ \beta & \gamma & \cdot \\ \cdot & \cdot & -\gamma \end{matrix} \right].
\]
If $z$ is an internal point of $\cC_z$ then we obtain the orbit
\[
\Sigma_{10} : \left[ \begin{matrix} \alpha & \beta & \cdot \\ \beta & \gamma & \cdot \\ \cdot & \cdot & -\varepsilon\gamma \end{matrix} \right], 
\quad \text{where } \epsilon \in \bF_q \text{ is a non-square}.
\]

\begin{lemma}\label{lem:Sigma_9_10}
Let $\pi$ be a plane in one of the $K$-orbits $\Sigma_9$ or $\Sigma_{10}$, and let $x$ be the unique point of rank $1$ in $\pi$. 
The points of rank $2$ in $\pi$ lie on the union of a line $\ell$ through $x$ and a non-degenerate conic meeting $\ell$ in the point $x$. 
The point-orbit distribution of $\pi$ is $[1,2q,0,q^2-q ]$ or $[1,q,q,q^2-q ]$ according to whether $\pi$ belongs to $\Sigma_9$ or $\Sigma_{10}$. 
\end{lemma} 

\begin{proof}
Suppose that $\pi$ is the representative of $\Sigma_9$ given above, with its unique point $x$ of rank~$1$ corresponding to $\beta=\gamma=0$. 
The points of rank~$2$ in $\pi$ lie on the cubic $\gamma(\alpha\gamma-\beta^2)=0$, which is the union of the line $\ell:\gamma=0$ and a non-degenerate conic $\cC:\alpha\gamma-\beta^2=0$ meeting $\ell$ in the point $x$. 
The points of rank~$2$ on $\ell$ belong to $\cP_{2,e}$ by Lemma~\ref{lem:extcriteria}. 
(Alternatively, observe that they lie on the tangent $t_x(\cC_y)$ to $\cC_y$ through $x$ in $\langle \cC_y \rangle$.) 
Each point of rank~$2$ in $\cC\setminus \ell$ has coordinates $(a^2,a,0,1,0,-1)$ for some $a \in \bF_q$ and so also belongs to $\cP_{2,e}$ by Lemma~\ref{lem:extcriteria}. 
The proof for $\Sigma_{10}$ is analogous, but now Lemma~\ref{lem:extcriteria} shows that the points of rank~$2$ in $\cC\setminus \ell$ belong to $\cP_{2,i}$.
\end{proof}
 
\begin{remark} \label{remarkSigma8910}
\textnormal{
It follows from Lemmas \ref{lem:Sigma_8} and \ref{lem:Sigma_9_10} that $\Sigma_8$, $\Sigma_9$ and $\Sigma_{10}$ are three distinct $K$-orbits. 
Indeed, if $\pi$ is a plane in one of these orbits then its orbit can be determined as follows. 
If the cubic $Q$ of points of rank at most $2$ in $\pi$ is the union of two lines, then $\pi \in \Sigma_8$. 
If not, then $Q$ is the union of a line $\ell$ and a non-degenerate conic $\cC$ meeting $\ell$ in the unique point $x$ of rank~$1$ in $\pi$. 
In this case, consider any line $\ell' \neq \ell$ in $\pi$ through $x$. 
Then $\ell'$ meets $\cC$ in a second point, which has rank~$2$. 
If this point belongs to $\cP_{2,e}$ then $\pi \in \Sigma_9$. 
Otherwise, $\pi \in \Sigma_{10}$.
}
\end{remark}

\subsection{Case {\bf (d)} }\label{ssec:cased}
Finally, suppose that $x\notin \cC_y\cup \cC_z$, and write $w = \cC_y\cap \cC_z$. 
Without loss of generality, we may fix $p_w$, $p_x$, $\ell_y$ and $\ell_z$. 
Let us take them as $\pt{e_1}$, $\pt{e_2+e_3}$, $\pt{e_1,e_2}$ and $\pt{ e_1,e_3}$, respectively. 
Then the stabiliser of this configuration in $\PG(2,q)$ is contained in the group of perspectivities with centre $p_w$ (it fixes three lines through $p_w$, and hence all lines through $p_w$).
On the plane $\langle \cC_y\rangle$ in $\PG(5,q)$, the induced group acts as the stabiliser of the conic $\cC_y$ and the point $w$. We consider separately the cases where (i) $y$ lies on the tangent line $t_w(\cC_y)$ to $\cC_y$ through $w$ in $\langle \cC_y\rangle$, and (ii) $y$ does not lie on $t_w(\cC_y)$.

\bigskip

{\bf (d-i)} Suppose that $y$ is on the tangent line $t_w(\cC_y)$ to $\cC_y$ through $w$ in $\langle \cC_y\rangle$. 
Let $u$ be the unique point on $\cC_y\backslash\{w\}$ such that $y=t_u(\cC_y)\cap t_w(\cC_y)$. 
We may take $p_u=\pt{e_2}$ and $y = (0,1,0,0,0,0)$.
Then the stabiliser of $x$, $y$, $w$ and $\cC_z$ is induced by the group of elations with centre $p_w$ and axis $\pt{p_x,p_u}$, and acts on $\langle \cC_z\rangle$ as the stabiliser of $\cC_z$ and the two points $w$ and 
\[
v:=\nu(\pt{p_x,p_u} \cap \ell_z)
\]
of $\cC_z$. 
Consider the tangent line $t_{v}(\cC_z)$ to $\cC_z$ through $v$ in $\langle \cC_z\rangle$. 
Note that $v=\nu(\pt{e_3})$.
There are five possibilities for $z$: 
(A) $z=t_w(\cC_y)\cap t_{v}(\cC_z)$, (B) $z \in t_w(\cC_y) \setminus t_v(\cC_z)$, (C) $z \in t_v(\cC_z) \setminus t_w(\cC_y)$, (D) $z \in \pt{w,v}$, and (E)~$z$~does not lie on $\pt{w,v}$. 

\bigskip

{\bf (d-i-A)} 
If $z=t_w(\cC_y)\cap t_v(\cC_z)$ then $z = (0,0,1,0,0,0)$. 
The plane $\pi=\pt{x,y,z}$ then also contains the point $r = (0,1,1,0,0,0)$, with $\cC_r=\nu(\langle p_x,p_w \rangle)$. 
Hence $\pi=\pt{x,r,z}$ with $x\in \cC_r$, $x\notin \cC_z$, and $z$ on the tangent $t_w(\cC_z)$ to $\cC_z$ through $w=\cC_r\cap\cC_z$ in $\langle \cC_z \rangle$.
This implies that $\pi$ belongs to the $K$-orbit $\Sigma_8$.

\bigskip

{\bf (d-i-B)} If $z$ lies on $t_w(\cC_y)$ but not on $t_{v}(\cC_z)$ then without loss of generality $z=(1,0,1,0,0,0)$, but then $\pi$ also contains the point $(1,1,1,1,1,1)$, which has rank~$1$, a contradiction.

\bigskip

{\bf (d-i-C)} If $z$ lies on $t_{v}(\cC_z)$ but not on $t_w(\cC_y)$ then without loss of generality $z=(0,0,1,0,0,1)$.
This yields a new orbit $\Sigma_{11}$ with the following representative:
\[
\Sigma_{11} : \left[ \begin{matrix} \cdot & \beta & \gamma \\ \beta & \alpha & \alpha \\ \gamma & \alpha & \alpha+\gamma \end{matrix} \right].
\]

\begin{lemma}\label{lem:Sigma_11}
A plane belonging to the $K$-orbit $\Sigma_{11}$ has point-orbit distribution $ [ 1 ,q,0, q^2]$.
\end{lemma} 

\begin{proof} The cubic of points of rank at most $2$ in the plane $\pi$ with the above representative has equation $\alpha(\beta-\gamma)^2-\beta^2\gamma=0$. 
Consider the lines through the point $x$ of rank $1$, which corresponds to $\beta=\gamma=0$. 
Each line through $x$ and a point with $\alpha=0$ and $\beta\neq \gamma$ contains exactly one point $s$ of rank~$2$, represented by the matrix
\[
M_s=\npmatrix{\cdot&\beta(\beta-\gamma)^2&\gamma(\beta-\gamma)^2\\
\beta(\beta-\gamma)^2& -\beta^2\gamma&-\beta^2\gamma \\
\gamma(\beta-\gamma)^2& -\beta^2\gamma&-\beta^2\gamma+\gamma(\beta-\gamma)^2}.
\]
In the notation of Lemma \ref{lem:extcriteria} we have $-|M_{11}(M_s)|=\beta^2\gamma^2(\beta-\gamma)^2$, $-|M_{22}(M_s)|=\gamma^2(\beta-\gamma)^2$ and $-|M_{33}(M_s)|=\beta^2(\beta-\gamma)^2$, implying that $s \in \cP_{2,e}$. 
This amounts to $q$ points of $\cP_{2,e}$. 
The line through $x$ and the point with $\alpha=0$ and $\beta=\gamma$ contains no points of rank~$2$.
\end{proof}

{\bf (d-i-D)} For convenience, in this section we change our points of reference, instead fixing $p_w$, $p_x$, $p_u$ and $p_v$ as $\pt{e_2}$, $\pt{e_1}$, $\pt{e_1+e_3}$ and $\pt{e_3}$, respectively. 
This implies that $\ell_y= \pt{e_1+e_3,e_2}$ and $\ell_z=\pt{ e_2,e_3}$. Note that $y$ is now the point with coordinates $(0,1,0,0,1,0)$. 
Since we are now assuming that $z$ is on the line $\pt{w,v}$, we may take $z=(0,0,0,1,0,b)$ for some non-zero $b \in \bF_q$. 
If $b$ is a square, say $b=1$, then we obtain the orbit 
\begin{equation} \label{Sigma12rep}
\Sigma_{12} : \left[ \begin{matrix} \alpha & \beta & \cdot \\ \beta & \gamma & \beta \\ \cdot & \beta & \gamma \end{matrix} \right].
\end{equation}
If $b$ is a non-square, we obtain the orbit
\begin{equation} \label{Sigma13rep}
\Sigma_{13} : \left[ \begin{matrix} \alpha & \beta & \cdot \\ \beta & \gamma & \beta \\ \cdot & \beta & \varepsilon\gamma \end{matrix} \right], 
\quad \text{where } \varepsilon \in \bF_q \text{ is a non-square}.
\end{equation}
Let $\pi_{12}$, $\pi_{13}$ denote these representatives, and let $C_{12}$, $C_{13}$ denote the respective cubic curves of points of rank at most $2$ (defined by setting the determinants of the representatives to zero). 

\begin{lemma}\label{lem:Sigma_12_13}
A plane belonging to the $K$-orbit $\Sigma_{12}$ has point-orbit distribution 
\[
[1,(q-1)/2,(q-1)/2,q^2+1],
\] 
and a plane belonging to the $K$-orbit $\Sigma_{13}$ has point-orbit distribution 
\[
[1,(q+1)/2,(q+1)/2,q^2-1 ].
\] 
\end{lemma} 

\begin{proof} 
The cubic curves $C_{i}$, $i\in \{12,13\}$, are given by $\alpha f_{i}(\beta,\gamma)+g_{i}(\beta,\gamma)=0$, where $f_{12}(\beta,\gamma)=\gamma^2-\beta^2$, $f_{13}(\beta,\gamma)=\epsilon\gamma^2-\beta^2$, $g_{12}=-\beta^2\gamma$ and $g_{13}=-\epsilon\beta^2\gamma$. 
Consider a line through the point $x$ of rank~$1$ (corresponding to $\beta=\gamma=0$) and a point with $\alpha=0$. 
If $f_i(\beta,\gamma) \neq 0$ then such a line contains a unique point of rank~$2$. 
Since $f_{13}(\beta,\gamma)=0$ has no non-trivial solutions, every line through $x$ in $\pi_{13}$ contains exactly one point of rank~$2$ and so $\pi_{13}$ has rank distribution $[1,q+1,q^2-1]$. 
On the other hand, $f_{12}(\beta,\gamma)=0$ if and only if $\beta = \pm\gamma$, and in these cases $g_{12}(\beta,\gamma) \neq 0$. 
Hence, in $\pi_{12}$ there are exactly two lines through $x$ which contain no point of rank~$2$, and so $\pi_{12}$ has rank distribution $[1,q-1,q^2+1]$. 
Now, the points of rank~$2$ in $\pi_{12}$ are those satisfying $\alpha=\beta^2\gamma/(\gamma^2-\beta^2)$, and Lemma \ref{lem:extcriteria} implies that such a point is in $\cP_{2,e}$ if and only if $\beta^2-\gamma^2$ is a non-zero square. 
By \cite[Theorem~64]{Dickson1901}, the quadratic form $\beta^2-\gamma^2$ evaluates to a  non-zero square for precisely $(q-1)^2/2$ inputs $(\beta,\gamma) \in \mathbb{F}_q^2$, so it follows (upon factoring out scalars) that $\pi_{12}$ contains $(q-1)/2$ points in $\cP_{2,e}$. 
Similarly a point of rank~$2$ in $\pi_{13}$ is in $\cP_{2,e}$ if and only if $\beta^2-\epsilon \gamma^2$ is a non-zero square. 
This occurs for $(q+1)(q-1)/2$ inputs $(\beta,\gamma) \in \mathbb{F}_q^2$, so it follows that $\pi_{13}$ contains $(q+1)/2$ points in $\cP_{2,e}$.
 \end{proof}

\begin{lemma} \label{lem:Sigma_12_13_cst_rk_3_line}
A plane in either of the $K$-orbits $\Sigma_{12}$ or $\Sigma_{13}$ contains a line of constant rank~$3$. 
\end{lemma} 

\begin{proof} 
The cubic curve of points of rank at most $2$ in such a plane does not form a blocking set, because it contains no lines and at most $q+2$ points. 
\end{proof}

\begin{lemma} \label{lem:Sigma_6_13}
The $K$-orbits $\Sigma_{6}$, $\Sigma_{12}$ and $\Sigma_{13}$ are distinct.
\end{lemma} 

\begin{proof} 
The $K$-orbits $\Sigma_{12}$ and $\Sigma_{13}$ are distinct by Lemma~\ref{lem:Sigma_12_13}. 
By Lemma \ref{lem:Sigma_6_2}, a plane in $\Sigma_6$ does not contain a line of constant rank $3$. 
Hence, $\Sigma_6$ is distinct from $\Sigma_{12}$ and $\Sigma_{13}$ by Lemma~\ref{lem:Sigma_12_13_cst_rk_3_line}.
\end{proof}

We also record the following lemma for future reference.

\begin{lemma} \label{lem:Sigma_13_o_13_line}
A plane in either of the $K$-orbits $\Sigma_{12}$ or $\Sigma_{13}$ contains a line of type $o_{13,1}$ or $o_{13,2}$. 
%A plane in the $K$-orbit $\Sigma_{12}$ (respectively, $\Sigma_{13}$) contains a line of type $o_{13,1}$ (respectively, $o_{13,2}$). 
\end{lemma} 

\begin{proof} 
Setting $\alpha=0$ in \eqref{Sigma12rep} and \eqref{Sigma13rep} yields the lines 
\[
\left[ \begin{matrix} \cdot & \beta & \cdot \\ \beta & \gamma & \beta \\ \cdot & \beta & \gamma \end{matrix} \right]
\quad \text{and} \quad 
\left[ \begin{matrix} \cdot & \beta & \cdot \\ \beta & \gamma & \beta \\ \cdot & \beta & \varepsilon\gamma \end{matrix} \right], 
\]
each of which is $K$-equivalent to either the $o_{13,1}$ representative or the $o_{13,2}$ representative from Table~\ref{table:lines}. 
(The exact $K$-orbit of each line depends on whether $-1$ is a square in $\mathbb{F}_q$ or not.)
\end{proof}

Before proceeding to the analysis of the case~(d-i-E), recall the definitions of the {\it Hessian} and the {\it points of inflexion} of a cubic curve: the Hessian is the determinant of the $3\times 3$ matrix of second derivatives, and the points of inflexion are the points of intersection of the Hessian and the curve. 
In the cases considered below, all points of inflexion lie on a common line, which we call the {\it line of inflexion}.
Cubic curves in characteristic $3$ require separate consideration, so we state some results only for characteristic $\neq 3$. 
(Lemma~\ref{lem:pi14char3} confirms that we do not need the any of the analogous results in characteristic $3$.) 

Consider again the planes $\pi_{12} \in \Sigma_{12}$ and $\pi_{13} \in \Sigma_{13}$ defined in \eqref{Sigma12rep} and \eqref{Sigma13rep}, respectively. 
Recall the definitions of the associated cubic curves $\overline{C}_{12}$ and $\overline{C}_{13}$ in $\PG(2,q)$ given immediately before Lemma~\ref{lem:Sigma_12_13}.
Given a fixed basis for $\pi_i$, $i \in \{12,13\}$, define a map from $\pi_i$ to $\PG(2,q)$ in the natural way, and denote the image of $C_i$ in $\PG(2,q)$ by $\overline{C}_i$.

\begin{lemma} \label{lem:Sigma_13_o_13_inflexion}
Suppose that $q$ is not a power of $3$. 
Then the cubic curves $\overline{C}_{12}$ and $\overline{C}_{13}$ each have a double point at $P=(1,0,0)$. 
The tangents at the double point, and the points and lines of inflexion, are as given in the following table:
\begin{center}
\begin{tabular}{lcc}
\toprule
Plane&$\pi_{12}$&$\pi_{13}$\\
\midrule
 Tangents at $P$&$\beta=\pm \gamma$&$\beta= \pm \sqrt{\varepsilon}\gamma$\\
Inflexion points&$\{(0,1,0),(-1,\pm 4\sqrt{-1/3},4)\}$&$\{(0,1,0),(-\varepsilon,\pm 4\sqrt{-\epsilon/3},4)\}$\\
Line of inflexion&$4\alpha+\gamma=0$&$4\alpha+\varepsilon\gamma=0$\\
\bottomrule
\end{tabular}
\end{center}
In particular, when $-3$ is a square in $\bF_q$, $\overline{C}_{12}$ has three $\bF_q$-rational inflexion points, and $\overline{C}_{13}$ has one. 
When $-3$ is a non-square, $\overline{C}_{12}$ has one $\bF_q$-rational inflexion point, and $\overline{C}_{13}$ has three.
\end{lemma} 

\begin{proof}
The cubic $\overline{C}_{12}$ is $\alpha(\gamma^2-\beta^2)-\beta^2\gamma$, so its Hessian is $8\alpha(\gamma^2 - \beta^2) +8 \gamma(2\beta^2 + \gamma^2)$, and the points of inflexion therefore satisfy $\gamma(3\beta^2+\gamma^2)=0$. 
A similar calculation applies for $\overline{C}_{13}$.
\end{proof}

{\bf (d-i-E)} If the point $z$ is not on the line $\pt{w,v}$ and not on the tangents $t_w(\cC_y)$ and $t_v(\cC_z)$, then we may assume the line $\pt{z,w}$ passes through a point $r\in\cC_z\setminus\{w,v\}$. 
We retain the same representatives as in case (d-i-D), and without loss of generality we further choose $r = \nu(p_r)$ with $p_r=\pt{e_2-e_3}$. 
We may then take $z=z_c:=(0,0,0,c,-1,1)$ where $c\not \in \{0,1\}$ (because $z_0$ lies on $t_v(\cC_z)$ and $z_1 = r$). 
The plane $\pi_c=\pt{x,y,z_c}$ is then represented by the matrix
\[
\left[ \begin{matrix} \alpha & \beta & \cdot \\ \beta & c\gamma & \beta-\gamma \\ \cdot & \beta-\gamma & \gamma \end{matrix} \right].
\]

\begin{lemma}\label{lem:nice}
Let $\ell$ be a line through the point $x$ in the plane $\pi_c$, $c\in \bF_q\setminus\{0,1\}$. 
Then $\ell$ is of type $o_{8,1}$ or $o_{8,2}$ unless $c$ is a square in $\bF_q$ and $\ell$ contains one of the two points $(0,\beta,0,c\gamma,\beta-\gamma,\gamma)$ with $\beta^2-2\beta\gamma+(1-c)\gamma^2=0$. 
In this case, $\ell$ is of type $o_9$. 
\end{lemma}

\begin{proof} 
The determinant of $\pi_c$ is $\alpha(\beta^2-2\beta\gamma+(1-c)\gamma^2)+\beta^2\gamma$. 
Therefore, a line $\ell$ in $\pi_c$ through the point $x=(1,0,0,0,0,0)$ and a point with coordinates $(0,\beta,0,c\gamma,\beta-\gamma,\gamma)$, $(\beta,\gamma)\neq (0,0)$, contains exactly one point of rank~$2$ provided that $\beta^2-2\beta\gamma+(1-c)\gamma^2\neq 0$. 
In this case, $\ell$ has type $o_{8,1}$ or $o_{8,2}$ because by Table~\ref{table:line-rank-dist} these are the only $K$-orbits of lines with rank distribution $[1,1,q-1]$.
If $\beta^2-2\beta\gamma+(1-c)\gamma^2 = 0$ then $c$ must be a square; in this case, $\ell$ contains no points of rank~$2$ and hence has type $o_9$. 
\end{proof}

\begin{lemma}\label{lem:pi14new}
Suppose that $q$ is not a power of $3$ and let $c\in \bF_q\setminus\{0,1\}$. 
If $-3c$ is a square in $\bF_q$ and $\frac{\sqrt{c}+1}{\sqrt{c}-1}$ is not a cube in $\bF_q(\sqrt{-3})$, then the plane $\pi_c$ is not in any of the $K$-orbits $\Sigma_1,\ldots,\Sigma_{13}$.
\end{lemma}

\begin{proof}
Note that we only need to show that $\pi_c \not \in \Sigma_{12} \cup \Sigma_{13}$.
Consider the cubic curve $C_{c}$ obtained by setting the determinant $\alpha(\beta^2-2\beta\gamma+(1-c)\gamma^2)+\beta^2\gamma$ of $\pi_{c}$ equal to zero. 
The points of inflexion are precisely the points of $C_c$ for which $(\beta,\gamma) \neq (0,0)$ and $(1-c)^2\gamma^3 - 3(1-c)\beta^2\gamma + 2\beta^3 = 0$. 
Since $c \neq 1$, any such point has $\beta \neq 0$, so we may rewrite this equation as $(1-c)^2\theta^3-3(1-c)\theta +2=0$, where $\theta = \gamma/\beta$. 
By \cite{Dickson1906}, this equation has no solutions in $\bF_q$ if and only if $-3c$ is a square in $\bF_q$ and  $\frac{\sqrt{c}+1}{\sqrt{c}-1}$ is not a cube in $\bF_q(\sqrt{-3})$. 
In this case, $C_c$ has no $\bF_q$-rational points of inflexion, and so, by Lemma~\ref{lem:Sigma_13_o_13_inflexion}, $\pi_c$ is not equivalent to a plane in $\Sigma_{12}$ or $\Sigma_{13}$.
\end{proof}

\begin{lemma} \label{lem:pi14equiv1213}
Let $c \in \bF_q \setminus \{0,1\}$, and suppose that either $-3c$ is a non-square in $\bF_q$ or $\frac{\sqrt{c}+1}{\sqrt{c}-1}$ is a cube in $\bF_q(\sqrt{-3})$. 
Then $\pi_c \in \Sigma_{12}$ if $c$ is a square in $\bF_q$, and $\pi_c \in \Sigma_{13}$ otherwise. 
\end{lemma}

\begin{proof}
We prove this by explicitly mapping $\pi_{12}$ or $\pi_{13}$ to $\pi_c$ (according to whether $c$ is a square or not). 
By \cite{Dickson1906}, the equation $4(1-c)^2v^3 + 3c(1 - c)v - c^2=0$ has a solution $v \in \bF_q$ if and only if $-3c$ is a non-square in $\bF_q$ or $\frac{\sqrt{c}+1}{\sqrt{c}-1}$ is a cube in $\bF_q(\sqrt{-3})$. 
If $c$ is a square in $\bF_q$, we may therefore define the matrix
\[
X = \left[ \begin{matrix} 1 & \frac{-zv}{v^2-z^2} & \frac{z^2}{v^2-z^2} \\ 
0 & v\sqrt{c} & z\sqrt{c} \\ 
0 & z & v \end{matrix} \right], 
\]
which satisfies $X\pi_{12}X^T = \pi_c$. 
Hence, $\pi_c \in \Sigma_{12}$ in this case. 
If $c$ is a non-square in $\bF_q$ then $\varepsilon c$ is a square, and so we can instead define
\[
X = \left[ \begin{matrix} 1 & \frac{-\epsilon zv}{\epsilon v^2-z^2} & \frac{z^2}{\epsilon v^2-z^2} \\ 
0 & v\sqrt{\varepsilon c} & z\sqrt{c/\varepsilon} \\ 
0 & z & v \end{matrix} \right].
\]
In this case, we have $X\pi_{13}X^T = \pi_c$ and hence $\pi_c \in \Sigma_{13}$.
\end{proof}

\begin{lemma}\label{lem:pi14char3}
Suppose that $q$ is a power of $3$, and let $c\in \bF_q\setminus\{0,1\}$. 
Then $\pi_c \in \Sigma_{12} \cup \Sigma_{13}$.
\end{lemma}

\begin{proof} 
Here every element of $\bF_q(\sqrt{-3})$ is a cube, so the assertion follows from Lemma~\ref{lem:pi14equiv1213}.
\end{proof}

\begin{lemma}
If $q$ is not a power of $3$ then there exists $c\in \bF_q\setminus\{0,1\}$ such that the plane $\pi_c$ is not in any of the $K$-orbits $\Sigma_1,\ldots,\Sigma_{13}$. 
\end{lemma}

\begin{proof}
By Lemma~\ref{lem:pi14new}, it suffices to show that there exists $c\in \bF_q$ such that $-3c$ is a square in $\bF_q$ and $\frac{\sqrt{c}+1}{\sqrt{c}-1}$ is not a cube in $\bF_q(\sqrt{-3})$.
The element $\frac{\sqrt{c}+1}{\sqrt{c}-1}$ is a cube in $\bF_q(\sqrt{-3})$ if and only if there exists $d\in \bF_q(\sqrt{-3})$ such that $\sqrt{c}=\frac{d^3+1}{d^3-1}$. 
The function $f(d)= \big(\frac{d^3+1}{d^3-1}\big)^2$ satisfies $f(d)=f(e)$ if and only if $d^3=e^3$ or $d^3=e^{-3}$, and is therefore six-to-one on $\bF_q(\sqrt{-3})\backslash\{d:d^7=d\}$. 
Furthermore, $f(d)$ is a non-zero square in $\bF_q$ if and only if $d^{3(q-1)}=1$ and $d^6\ne 1$, and a non-square in $\bF_q$ if and only if $d^{3(q+1)}=1$ and $d^6\ne 1$.
If $q\equiv 1\pmod 3$ then $\bF_q(\sqrt{-3})=\bF_q$, and $f(d)$ is a non-zero square for every $d\in \bF_q$ such that $d^7\ne d$. 
Hence, there are $\frac{q-7}{6}+1$ elements of $\bF_q$ in the image of $f$, so there exists $c\in \bF_q$ such that $c$ (and hence $-3c$) is a square in $\bF_q$ and $\frac{\sqrt{c}+1}{\sqrt{c}-1}$ is a non-cube in $\bF_q(\sqrt{-3})$.
If $q\equiv 2\pmod 3$ then $\bF_q(\sqrt{-3})=\bF_{q^2}$. 
Since $\gcd(q^2-1,3(q+1))=q-1$, there are $q-5$ solutions to $d^{3(q+1)}=1$, $d^6\ne 1$, so there are $\frac{q-5}{6}$ non-squares of $\bF_q$ in the image of $f$. 
Hence, there exists $c\in \bF_q$ such that $c$ is a non-square (so $-3c$ is a square) in $\bF_q$ and $\frac{\sqrt{c}+1}{\sqrt{c}-1}$ is a non-cube in $\bF_q(\sqrt{-3})$.
\end{proof}

\begin{lemma} \label{lem:sigma14equiv}
Suppose that $q$ is not a power of $3$ and let $c,d \in \bF_q\setminus\{0,1\}$ such that $-3c$ and $-3d$ are both squares in $\bF_q$. 
Then the planes $\pi_{c}$ and $\pi_{d}$ belong to the same $K$-orbit if one of the following conditions holds: 
\begin{itemize}
\item[(i)] $cd=1$, 
\item[(ii)] $\big(\frac{\sqrt{c}-1}{\sqrt{c}+1}\big)\big(\frac{\sqrt{d}-1}{\sqrt{d}+1}\big)$ is a cube in $\bF_q(\sqrt{-3})$, 
\item[(iii)] $\big(\frac{\sqrt{c}+1}{\sqrt{c}-1}\big)\big(\frac{\sqrt{d}-1}{\sqrt{d}+1}\big)$ is a cube in $\bF_q(\sqrt{-3})$.
\end{itemize}
\end{lemma}

\begin{proof}
If $cd=1$ then $X\pi_{c}X^T=\pi_{d}$ for 
\[
X = \left[ \begin{matrix} 1 & -d & -1 \\ 0 & 0 & -d \\ 0 & -d & 0 \end{matrix} \right], 
\]
so $\pi_c$ and $\pi_d$ belong to the same $K$-orbit under condition~(i).
Now consider conditions~(ii) and~(iii). 
Since $-3c$ and $-3d$ are both squares in $\bF_q$, so are $cd$ and $d/c$. 
Hence, we may instead define 
\[
X = \left[ \begin{matrix} 1 & \pm x_{12} & x_{13} \\ 0 & \pm\theta\sqrt{d/c} & \psi \\ 0 & \pm\psi/\sqrt{cd} & \theta \end{matrix} \right]
\]
for some $x_{12}$, $x_{13}$, $\psi$, $\theta$. 
To satisfy $X\pi_{c}X^T=\pi_{d}$, we require that
\[
x_{12} = \frac{\psi^3 + 2\psi^2\theta +d\psi\theta^2}{\sqrt{cd}(\psi^2 - d\theta^2)}, 
\quad 
x_{13} = -\theta-1-\frac{2\theta^2(\psi+d\theta)}{\psi^2-d\theta},
\]
and that $\psi$, $\theta$ are common solutions to 
\begin{align*}
0 &= \psi^2+2\psi\theta\pm\sqrt{d/c} \psi +d(\theta^2-\theta), \\
0 &= \psi^2/d+2\psi\theta+\psi +\theta^2-\theta\sqrt{d/c}.
\end{align*}
By calculating the resultant of the above polynomials with respect to $\psi$, we find that there is a solution if and only if 
\[
\theta^3-\frac{3d}{4c}\left(\frac{c-1}{d-1}\right)\theta-\frac{d}{4c}\sqrt{\frac{d}{c}}\left(\frac{(\sqrt{cd}\mp1)(c-1)}{(d-1)^2}\right) = 0
\]
has a solution $\theta \in \bF_q$. 
The discriminant of the above cubic is 
\[
(-3d)\left(\frac{3d(c-1)(1\mp\sqrt{d/c})}{4c(d-1)^2}\right)^2,
\]
which is a square in $\bF_q$ because $-3d$ and $d/c$ are both squares. 
Hence, by \cite{Dickson1906}, the cubic has a solution in $\bF_q$ if and only if 
\[
\frac{d\sqrt{d}(c-1)(\sqrt{d}+1)(\sqrt{c}\mp 1)}{8c\sqrt{c}(d-1)^2}
\]
is a cube in $\bF_q(\sqrt{-3})$, which is precisely when $\big(\frac{\sqrt{c}\pm 1}{\sqrt{c}\mp 1}\big)\big(\frac{\sqrt{d}-1}{\sqrt{d}+1}\big)$ is a cube in $\bF_q(\sqrt{-3})$.
\end{proof}

\begin{lemma} \label{Sigma14lemma}
All planes $\pi_c$ such that $\pi_c \notin \Sigma_1 \cup \dots \cup \Sigma_{13}$ belong to the same $K$-orbit. 
\end{lemma}

\begin{proof}
By Lemma \ref{lem:pi14char3}, we may assume that $q$ is not a power of $3$. 
Suppose that $\pi_c$ and $\pi_d$ are two such planes. 
By Lemma~\ref{lem:pi14equiv1213}, $-3c$ and $-3d$ are both squares in $\bF_q$, so $cd$ is also a square in $\bF_q$.
Moreover, $\frac{\sqrt{c}+1}{\sqrt{c}-1}$ and $\frac{\sqrt{d}+1}{\sqrt{d}-1}$ are both non-cubes in $\bF_q(\sqrt{-3})$. 
Hence, $\frac{\sqrt{c}+1}{\sqrt{c}-1} = \omega^i c_1^3$ and $\frac{\sqrt{d}+1}{\sqrt{d}-1}=\omega^j d_1^3$ for some $c_1,d_1\in \bF_q(\sqrt{-3})$ and some $i,j\in \{1,2\}$, where $\omega$ is a primitive third root of unity. 
If $i=j$ then $\big(\frac{\sqrt{c}+1}{\sqrt{c}-1}\big)\big(\frac{\sqrt{d}-1}{\sqrt{d}+1}\big) = \big(\frac{d_1}{c_1}\big)^3$ is a cube in $\bF_q(\sqrt{-3})$, so $\pi_c$ and $\pi_d$ belong to the same $K$-orbit by Lemma~\ref{lem:sigma14equiv}(iii). 
If $i\ne j$ then $\big(\frac{\sqrt{c}-1}{\sqrt{c}+1}\big)\big(\frac{\sqrt{d}-1}{\sqrt{d}+1}\big)= \big(\frac{1}{c_1d_1}\big)^3$ is a cube in $\bF_q(\sqrt{-3})$, so $\pi_c$ and $\pi_d$ belong to the same $K$-orbit by Lemma \ref{lem:sigma14equiv}(ii).
\end{proof}

We denote the $K$-orbit arising from Lemma~\ref{Sigma14lemma} by $\Sigma_{14}$:
\[
\Sigma_{14}: \left[ \begin{matrix} \alpha& \beta & \cdot \\ \beta & c\gamma & \beta-\gamma \\ \cdot & \beta-\gamma & \gamma \end{matrix} \right], 
\]
where $q \not \equiv 0 \pmod 3$, $c \in \bF_q \setminus \{0,1\}$, $-3c$ is a square in $\bF_q$, and $\tfrac{\sqrt{c}+1}{\sqrt{c}-1}$ is a not a cube in $\mathbb{F}_q(\sqrt{-3})$.

\begin{lemma}
A plane belonging to the $K$-orbit $\Sigma_{14}$ has point-orbit distribution 
\[
[1, (q \mp 1)/2, (q \mp 1)/2, q^2 \pm 1], \quad \text{where } q \equiv \pm 1 \pmod 3.
\]
\end{lemma}

\begin{proof}
Denote the above representative of $\Sigma_{14}$ by $\pi_{14}$. 
Consider the cubic curve defined by setting the determinant of $\pi_{14}$ equal to zero. 
The tangents through the double point $P:\beta=\gamma=0$ are $\beta=(1+\sqrt{c})\gamma$ and $\beta=(1-\sqrt{c})\gamma$, which are $\bF_q$-rational if and only if $c$ is a square in $\bF_q$. 
If $q\equiv 1 \pmod 3$ then $-3$ is a square, and thus $c$ is a square. 
Hence, there are $q-1$ points of rank $2$ in $\pi_{14}$. 
If $q\equiv -1 \pmod 3$ then $-3$ is a non-square, and thus $c$ is a non-square. 
In this case there are $q+1$ points of rank $2$ in $\pi_{14}$.
By Lemma~\ref{lem:extcriteria}, a rank-$2$ point of $\pi_{14}$ with $\beta\ne 0$ is exterior if and only if $(c-1)\gamma^2+2\beta\gamma-\beta^2$ is a non-zero square. 
This occurs for $(q+1)/2$ choices of $\beta,\gamma$ if $c$ is a non-square, and $(q-1)/2$ choices of $\beta,\gamma$ if $c$ is a square.
\end{proof}

{\bf (d-ii)} 
The only planes $\pi=\pt{x,y,z}$ with $\operatorname{rank}(x)=1$ and $\operatorname{rank}(y)=\operatorname{rank}(z)=2$ that we have not yet considered are those such that $y$ is not on the tangent $t_w(\cC_y)$ to the conic $\cC_y$ through the point $w=\cC_y\cap \cC_z$ in $\langle \cC_y \rangle$ (and likewise $z$ is not on $t_w(\cC_z)$).
We show that there exist such planes not belonging to any of the previously considered $K$-orbits only in the case $q \equiv 0 \pmod 3$, and that all of those planes form a single $K$-orbit.

We begin with a lemma concerning certain lines spanned by points of rank $2$. 
(Here the notation is as above, but $y$ and $z$ are {\em arbitrary} points of rank $2$.)

\begin{lemma} \label{lem:o1314}
Suppose that $\ell$ is a line in $\PG(5,q)$ spanned by two points $y$ and $z$ of rank $2$ such that $\cC_y \ne \cC_z$. 
Then $\ell$ is of type $o_{12}$, $o_{13,1}$, $o_{13,2}$, $o_{14,1}$ or $o_{14,2}$. 
Specifically, 
\begin{itemize}
\item $\ell$ is of type $o_{12}$ if and only if $y\in t_w(\cC_y)$ and $z\in t_w(\cC_z)$, 
\item $\ell$ is of type $o_{13,1}$ or $o_{13,2}$ if and only if $y\in t_w(\cC_y)$ and $z\notin t_w(\cC_z)$, 
\item $\ell$ is of type $o_{14,1}$ or $o_{14,2}$ if and only if $y\notin t_w(\cC_y)$ and $z\notin t_w(\cC_z)$, 
\end{itemize}
where $w= \cC_y\cap \cC_z$.
Furthermore, if $\ell$ is of type $o_{14,1}$ or $o_{14,2}$ then $\ell$ contains a third point $u$ of rank $2$, and at least one of $y$, $z$, $u$ is an external point to its respective conic $\cC_y$, $\cC_z$, $\cC_u$. 
Furthermore, $y\in \pt{\cC_y\cap \cC_z,\cC_y\cap \cC_u}$, $z\in \pt{\cC_y\cap \cC_z,\cC_z\cap \cC_u}$ and $u\in \pt{\cC_y\cap \cC_u,\cC_z\cap \cC_u}$.
\end{lemma}

\begin{proof}
The hyperplane spanned by the distinct conics $\cC_y$ and $\cC_z$ intersects $\cV(\bF_q)$ in the union of these two conics, so the line $\ell = \langle y,z \rangle$ does not intersect $\cV(\bF_q)$. 
Table~\ref{table:line-rank-dist} then implies that $\ell$ must have type $o_{10}$, $o_{12}$, $o_{13,1}$, $o_{13,2}$, $o_{14,1}$ or $o_{14,2}$. 
Type $o_{10}$ is ruled out by observing that a line of type $o_{10}$ lies in a conic plane, so any two points $y$ and $z$ of rank~$2$ on such a line have $\cC_y = \cC_z$, contradicting our assumption. 
The assertions about the remaining types may be verified by direct calculations using the representatives in Table~\ref{table:lines}.
\end{proof}

\begin{lemma}\label{lem:all_o_14}
Suppose that $\pi$ is a plane in $\PG(5,q)$ containing a point of rank~$1$ and spanned by points of rank at most $2$, with $\pi \not \in \Sigma_1 \cup \ldots \cup \Sigma_{14}$. 
If $\ell$ is a line in $\pi$ spanned by points of rank~$2$ and not containing a point of rank~$1$, then $\ell$ has type $o_{14,1}$ or $o_{14,2}$.
\end{lemma}

\begin{proof}
The result follows from Lemma~\ref{lem:o1314} (and the remarks preceding it).
\end{proof}

Continuing with the above notation, we may now assume that all lines in $\pi$ spanned by points of rank~$2$ and not containing a point of rank~$1$ are of type $o_{14,1}$ or $o_{14,2}$. 
Without loss of generality, we may choose $p_x=(1,0,0)$, $p_w=(0,0,1)$, $\ell_y:X_0=0$ and $\ell_z:X_0=X_1$, where $(X_0,X_1,X_2)$ are the homogeneous coordinates in $\PG(2,q)$. 
Since $\pt{y,z}$ is of type $o_{14,1}$ or $o_{14,2}$, Lemma~\ref{lem:o1314} implies that $y\in \pt{w,y_1}$ and $z\in\pt{w,z_1}$ for some $y_1\in \cC_y$ and $z_1\in\cC_z$.
There are two cases to consider: (A) $p_{z_1} = \pt{p_x,p_{y_1}}\cap \ell_z$ and (B) $p_{z_1} \neq \pt{p_x,p_{y_1}}\cap \ell_z$.

\bigskip

{\bf (d-ii-A)} Here we may take $p_{z_1} = (1,1,0)$, and so $\pi$ is represented by the matrix
\[
\left[ \begin{matrix} \alpha+\gamma & \gamma & \cdot \\ \gamma & \beta+\gamma & \cdot \\ \cdot & \cdot & b\beta+c\gamma \end{matrix} \right]
\]
for some $b,c\in \bF_q$, where $y = (0,0,0,1,0,b)$ and $z = (1,1,0,1,0,c)$. 
However, by choosing $\beta,\gamma$ such that $b\beta+c\gamma=0$, we then obtain another point $u\in \pt{y,z}$ of rank~$2$ such that $x\in \cC_u$. 
Hence, this case has already been considered (see the discussion at the beginning of Section~\ref{ssec:cased}).

\bigskip

{\bf (d-ii-B)} In this case we may take $p_{z_1} = (1,1,1)$. 
We then have $\pi = \pi_{b,c}$ for some $b,c\in \bF_q$, where $\pi_{b,c}$ is the plane represented by the matrix 
\[
A_{b,c} = \left[ \begin{matrix} \alpha+\gamma & \gamma & \gamma \\ \gamma & \beta+\gamma & \gamma \\ \gamma & \gamma & b\beta+c\gamma \end{matrix} \right]. 
\]
In other words, $y = (0,0,0,1,0,b)$ and $z = (1,1,1,1,1,c)$. 
Note that if $c=1$ then $z$ has rank~$1$, and if $b=0$ then $y$ has rank~$1$, so we assume that $b(c-1)\ne 0$. 

\begin{lemma}\label{lem:nice2}
Suppose that $b\neq 0$ and $c\neq 1$, and let $\ell$ be a line in $\pi_{b,c}$ through the point $x = (1,0,0,0,0,0)$. 
Then $\ell$ is of type $o_{8,1}$ or $o_{8,2}$, unless $\ell$ contains a point with coordinates in 
\[
\{(\gamma,\gamma,\gamma,\beta+\gamma,\gamma, b\beta+c\gamma) : (\beta,\gamma)\in \PG(1,q), \; b\beta^2+(b+c)\beta\gamma+(c-1)\gamma^2=0\},
\]
in which case $\ell$ is of type $o_9$. 
\end{lemma}

\begin{proof}
The determinant of the matrix $A_{b,c}$ is $\alpha f_{b,c}(\beta,\gamma)+g_{b,c}(\beta,\gamma)$, where
\[
f_{b,c}(\beta,\gamma) = b\beta^2+(b+c)\beta\gamma+(c-1)\gamma^2, \quad 
g_{b,c}(\beta,\gamma) = \beta\gamma(b\beta+(c-1)\gamma).
\]
Since $b\neq 0$ and $c\ne 1$, the zero locus $\cZ(g_{b,c})$ of $g_{b,c}$ consists of three distinct points in $\PG(1,q)$, and none of these points lies on $\cZ(f_{b,c})$. 
Hence, if $(\beta_0,\gamma_0)\in \PG(1,q)$ with $f_{b,c}(\beta_0,\gamma_0)=0$, then $g_{b,c}(\beta_0,\gamma_0) \neq 0$. 
For such $(\beta_0,\gamma_0)$, the line through $x$ and the point of $\pi_{b,c}$ parameterised by $(0,\beta_0,\gamma_0)$ contains no points of rank~$2$, and is therefore a line of type $o_9$ (by Table~\ref{table:line-rank-dist}).
If $f_{b,c}(\beta_0,\gamma_0)\neq 0$ then this line contains exactly one point of rank~$2$, so is of type $o_{8,1}$ or $o_{8,2}$. 
\end{proof}

\begin{lemma}\label{lem:o13}
If the plane $\pi_{b,c}$ does not belong to any of the $K$-orbits $\Sigma_1,\dots,\Sigma_{14}$ then $\pi_{b,c}$ contains exactly $q$ points of rank~$2$.
\end{lemma}

\begin{proof}
Consider the quadratic form $f_{b,c}$ defined in the proof of Lemma~\ref{lem:nice2}. 
Then $\cZ(f_{b,c})$ consists of either zero, one or two points in $\PG(1,q)$, and in these respective cases, the plane $\pi_{b,c}$ contains $q+1$, $q$, or $q-1$ points of rank~$2$. 
Choose a point $y\in \pi_{b,c}$ of rank~$2$. 
Since any line through $y$ and another point of rank~$2$ is of type $o_{14,1}$ or $o_{14,2}$, every such line contains three points of rank~$2$ (see Table~\ref{table:line-rank-dist}). 
Hence, such lines partition the set of rank-$2$ points other than $y$ into pairs, and so the number of points of rank $2$ in $\pi_{b,c}$ is odd and therefore equal to $q$.
\end{proof}

\begin{lemma}\label{lem:o13b}
If $q$ is not a power of $3$ then every plane that intersects $\cV(\bF_q)$ and is spanned by points of rank at most $2$ belongs to one of the $K$-orbits $\Sigma_1,\dots,\Sigma_{14}$.
\end{lemma}

\begin{proof} 
We prove the contrapositive. 
We have already established that any such plane {\em not} belonging to any of the $K$-orbits $\Sigma_1,\dots,\Sigma_{14}$ is equivalent to a plane $\pi_{b,c}$ with $b\neq 0$ and $c\neq 1$. 
Using again the notation from the proof of Lemma \ref{lem:nice2}, observe that the quadric $\cZ(f_{b,c})$ in $\PG(1,q)$ consists of one point if and only if $(b-c)^2 = -4b$. 
By Lemma~\ref{lem:o13} (and its proof), this must be the case.
Consider the line $\ell$ of $\pi_{b,c}$ parameterised by $(\alpha,\beta,\gamma)$ with $\alpha+\gamma=0$. 
Note that $\ell$ does not contain the point $x : \beta=\gamma=0$ of rank~$1$, so Lemma~\ref{lem:all_o_14} implies that $\ell$ either contains at most one point of rank~$2$, or has type $o_{14,1}$ or $o_{14,2}$, in which case it contains exactly three points of rank~$2$ (by Table~\ref{table:line-rank-dist}). 
The points of rank~$2$ on $\ell$ are determined by the solutions of 
$\gamma^2((c-1)\gamma+(b+1)\beta)=0$, so $\ell$ contains exactly two points of rank~$2$ unless $b=-1$. 
Therefore, we must have $b=-1$. 
Putting this into $(b-c)^2 = -4b$ yields $c=-3$ (because $c \neq 1$). 
However, then the points of rank~$2$ on the line of $\pi_{-1,-3}$ parameterised by $(\alpha,\beta,\gamma)$ with $\alpha+\beta=0$ are determined by the solutions of $\beta^2(3\gamma+\beta)=0$.  
This line, similarly, cannot contain exactly two points of rank~$2$, and so $q$ must be a power of $3$.
\end{proof}

By Lemma \ref{lem:o13b} and its proof, if the plane $\pi_{b,c}$ does {\em not} belong to one of the previously considered $K$-orbits $\Sigma_1,\dots,\Sigma_{14}$, then $q$ is a power of $3$ and 
$(b,c)=(-1,0)$. 
This yields the following representative of a new orbit which we call $\Sigma_{14}'$:
\[
\Sigma_{14}': \left[ \begin{matrix} \alpha+\gamma& \gamma & \gamma \\ \gamma & \beta+\gamma & \gamma \\ \gamma & \gamma & -\beta \end{matrix} \right], 
\quad \text{for } q \equiv 0 \; (\text{mod } 3).
\]

\begin{lemma} \label{Sigma14'-distribution}
If $q$ is a power of $3$ then the $K$-orbit $\Sigma_{14}'$ is distinct from the $K$-orbits $\Sigma_1,\dots,\Sigma_{14}$, and a plane in $\Sigma_{14}'$ has point-orbit distribution $[1,q,0,q^2]$. 
\end{lemma}

\begin{proof}
Let $\pi$ denote the above representative of $\Sigma_{14}'$. 
The points of rank~$2$ in $\pi$ have the form 
\[
\left[ \begin{matrix} \frac{\gamma^3}{(\beta-\gamma)^2} & \gamma & \gamma \\ \gamma & \beta+\gamma & \gamma \\ \gamma & \gamma & -\beta \end{matrix} \right] 
\quad \text{with} \quad \beta \neq \gamma.
\]
The three principal minors of this matrix are 
\[
-\left(\frac{\gamma\beta}{\beta-\gamma}\right)^2, \; -(\gamma-\beta)^2, \; -\left(\frac{\gamma(\beta+\gamma)}{\beta-\gamma}\right)^2.
\] 
Hence, by Lemma~\ref{lem:extcriteria}, all points of rank~$2$ in $\pi$ are in $\cP_{2,e}$, and there are $q$ such points, as claimed. 
Table~\ref{table:pt-orbit-dist} now implies that $\pi$ does not belong to any of the previously considered $K$-orbits, with the possible exception of $\Sigma_{11}$. 
Therefore, it remains to show that $\pi \not \in \Sigma_{11}$. 
By Lemma~\ref{lem:all_o_14}, every line in $\pi$ spanned by points of rank~$2$ and not containing the unique point of rank~$1$ in $\pi$ has type $o_{14,1}$ or $o_{14,2}$. 
In particular, every such line contains exactly three points of rank~$2$. 
Consider, on the other hand, the representative of $\Sigma_{11}$ from Table~\ref{table:main}, namely
\[
\left[ \begin{matrix} \cdot & \beta & \gamma \\ \beta & \alpha & \alpha \\ \gamma & \alpha & \alpha+\gamma \end{matrix} \right].
\]
The points in this plane parameterised by $(\alpha,\beta,\gamma) = (0,1,0)$ and $(0,0,1)$ both have rank~$2$ and span a line with rank distribution $[0,2,q-1]$. 
Therefore, $\pi \not \in \Sigma_{11}$, as claimed. 
\end{proof}

%SECTION
\section{Planes meeting $\cV(\bF_q)$ in a point and not spanned by points of rank $\le 2$}\label{sec:final}

Finally, we consider the planes in $\PG(5,q)$ which contain exactly one point of rank~$1$ and are not spanned by points of rank at most $2$. 
Let $\pi$ be such a plane, and let $x$ denote the unique point of rank~$1$ in $\pi$. 
By assumption, at most one of the lines through $x$ in $\pi$ contains another point of rank at most $2$, so the other lines through $x$ contain no points of rank $2$. 
By inspecting the possible rank distributions in Table~\ref{table:line-rank-dist}, we see that each such line has type $o_9$. 
Hence, without loss of generality we can assume that $\pi$ contains the representative of the line orbit $o_9$ given in Table~\ref{table:lines}. 
In other words, we can assume that $x=(1,0,0,0,0,0)$ and that $\pi$ contains the rank-$3$ point $z = (0,0,1,1,0,0)$. 
We may then assume that $\pi = \langle x,y,z \rangle$ where $y$ has coordinates $(0,1,0,a,b,c)$ for some $a,b,c \in \mathbb{F}_q$. 
Hence, we may represent $\pi$ by the matrix 
\[
\left[ \begin{matrix} \alpha& \beta & \gamma \\ \beta & a\beta+\gamma & b\beta \\ \gamma & b\beta & c\beta \end{matrix} \right].
\]
The determinant of this matrix is $\alpha f(\beta, \gamma) + g(\beta, \gamma)$, where $f(\beta,\gamma) = (ac-b^2)\beta^2 + c\beta\gamma$ and $g(\beta,\gamma) = \beta^2(2b\gamma-c\beta) - \gamma^2(a\beta+\gamma)$. 
There is at most one point $\pt{(\beta_0,\gamma_0)}$ of $\PG(1,q)$ such that $f(\beta_0,\gamma_0)\ne 0$, because any such point corresponds to a point in $\pi$ of rank at most $2$, namely the point parameterised by $(\alpha_0,\beta_0,\gamma_0)$ with $\alpha_0=-{g(\beta_0,\gamma_0)}/{f(\beta_0,\gamma_0)}$. 
Since $f$ defines a quadric on $\PG(1,q)$, this implies that $f(\beta,\gamma)$ must be identically zero. 
Therefore, $b=c=0$.
We claim that $a=0$ also. 
If not, then the points in $\pi$ of rank $2$ are those parameterised by $(\alpha,\beta,\gamma)$ with $\gamma^2(a\beta+\gamma)=0$, contradicting the fact that the points of rank at most $2$ of lie on a line. 
Hence, $a=0$. 
This yields the following representative of the final $K$-orbit, $\Sigma_{15}$: 
\[
\Sigma_{15} : \left[ \begin{matrix} \alpha& \beta & \gamma \\ \beta & \gamma & \cdot \\ \gamma & \cdot & \cdot \end{matrix} \right].
\]

\begin{lemma}\label{lem:Sigma_?}
The $K$-orbit $\Sigma_{15}$ is distinct from all of $\Sigma_1, \dots, \Sigma_{14}$, and from $\Sigma_{14}'$ when $q \equiv 0 \pmod 3$. 
A plane in $\Sigma_{15}$ has point-orbit distribution $[1,q,0,q^2]$.
\end{lemma} 

\begin{proof}
The points of rank at most $2$ in the above representative of $\Sigma_{15}$ are all on the line $\langle x,y\rangle$ (namely $\gamma=0$), which has type $o_6$ (see Table~\ref{table:lines}). 
Moreover, the points of rank $2$ all lie in $\cP_{2,e}$ (by Lemma~\ref{lem:extcriteria}). 
Hence, this plane has point-orbit distribution $[1,q,0,q^2]$. 
As explained at the beginning of this section, the plane is not in any of the previously considered $K$-orbits (because it contains a single point of rank~$1$ and is not spanned by points of rank at most $2$).
\end{proof}

This completes the proof of Theorem~\ref{mainthm}.

%SECTION
\section{Wilson's classification of nets of rank one}\label{sec:Wilson}

As explained in Section~\ref{sec:prelims} (see Proposition~\ref{prop:nets_vs_planes}), Theorem~\ref{mainthm} implies Corollary~\ref{maincor}, namely the classification of nets of conics of rank one in $\PG(2,q)$ (for $q$ odd).
We now compare the latter classification with that of Wilson~\cite{Wilson1914}, and explain why Wilson's classification was incomplete. 

In Part~I of his paper~\cite{Wilson1914}, Wilson studied the nets of rank one, namely those containing a repeated line. 
He obtained 11 ``canonical types", labelled $\mathrm{I,II,\ldots, XI}$.
In Part II, he studied the nets of rank two (those not containing a repeated line, but containing a conic which is not absolutely irreducible), obtaining six canonical types $\mathrm{XII,\ldots, XVII}$. 
In Part III, he obtained two canonical types $\mathrm{XVIII}$ and $\mathrm{XIX}$ of nets of rank three.
Wilson was aware of the fact that he had not completely classified the orbits, pointing out that {\em ``All questions of inter-relations between these types have been considered and answered, except with respect to the two cases, nets XVI and XVII, and nets XVIII and XIX."} \cite[p. 207]{Wilson1914}.

\begin{table}[!ht]
\begin{tabular}{c|c|c|c|c|c|c|c|c|c|c|c|c|c|c|c}
$\Sigma_1$ & $\Sigma_2$ & $\Sigma_3$	& $\Sigma_4$ & $\Sigma_5$ & $\Sigma_6$ & $\Sigma_7$ & $\Sigma_8$	& $\Sigma_9$ & $\Sigma_{10}$ & $\Sigma_{11}$ & $\Sigma_{12}$ & $\Sigma_{13}$ & $\Sigma_{14}$ & $\Sigma'_{14}$ & $\Sigma_{15}$ \\
\hline
& $\mathrm{III}$ & $\mathrm{IV}$ & $\mathrm{VI}$ & $\mathrm{XI}$	& $\mathrm{VIII}$ & $\mathrm{I}$ & $\mathrm{V}$ & $\mathrm{VII}$ & $\mathrm{VII}$ & $\mathrm{IX}$ & $\mathrm{XI}$ & $\mathrm{X,XI}$ & $\mathrm{X,XI}$ & & $\mathrm{II}$ \\
\end{tabular}
\caption{Correspondence between $K$-orbits of planes in $\PG(5,q)$ and canonical types $\mathrm{I},\ldots,\mathrm{XI}$ of nets of conics of rank one in $\PG(2,q)$ obtained by Wilson~\cite{Wilson1914}.}
\label{Wilson-table}
\end{table}

Although Wilson's work was in general very thorough, there are also some other issues with his classification. 
Besides the aforementioned open cases, there are some inaccuracies, and some missing orbits.
Table~\ref{Wilson-table} shows the correspondence between the $K$-orbits $\Sigma_1,\ldots, \Sigma_{15}$ of planes in $\PG(5,q)$ obtained in this paper, and Wilson's canonical types $\mathrm{I},\ldots,\mathrm{XI}$ of nets of conics of rank one in $\PG(2,q)$.
As the table illustrates, neither the $K$-orbit $\Sigma_1$ nor the $K$-orbit $\Sigma'_{14}$ (which only arises in characteristic $3$) appear in Wilson's classification. 
Moreover, some of Wilson's canonical types of nets correspond to more than one $K$-orbit. 
Specifically, type $\mathrm{VII}$ is the union of the $K$-orbits $\Sigma_9$ and $\Sigma_{10}$, and type $\mathrm{XI}$ includes the $K$-orbits $\Sigma_5$, $\Sigma_{12}$, and sometimes $\Sigma_{14}$.
The equivalence between the nets of types $\mathrm X$ and $\mathrm{XI}$ is studied on p.~194 of Wilson's paper, and on p.~196 the author concludes that these types are equivalent except when $q$ has the form $36k+17$. 
Taking into account that the $K$-orbit $\Sigma_1$ was overlooked and that the nets of type $\mathrm{VII}$ split into two $K$-orbits, this would imply that for such values of $q$ the number of $K$-orbits of planes corresponding to nets of rank one would be $14$, contradicting our classification (in which there are $15$ orbits for every value of $q$).

\section*{Acknowledgements}

The first author acknowledges the support of {\em The Scientific and Technological Research Council of Turkey} T\"UB\.{I}TAK (project no.~118F159).

\end{document}